\documentclass[letterpaper,11pt]{amsart}
\usepackage[margin=1in]{geometry}

\usepackage{latexsym,amsfonts,amssymb, amstext, mathrsfs, framed}
\usepackage{ytableau}
 \usepackage{graphicx,tikz, subcaption, tikz-cd}
\usepackage{mathrsfs,hyperref, amssymb}
\usepackage[utf8]{inputenc}
\usepackage[all]{xy}
\usetikzlibrary{decorations.markings, decorations.pathmorphing, arrows.meta, calc, bending}
\usepackage{amsmath}
\usepackage{amsthm}
\usepackage{psfrag}
\usepackage{epsfig}
\usepackage{enumitem}
\usepackage{tikz}
\usepackage{caption, wrapfig}
\usepackage{csquotes}

\tikzstyle{vertex}=[circle, draw, inner sep=0pt, minimum size=6pt]

\usepackage{xcolor} 
\usepackage[most]{tcolorbox}

 \newcommand{\uq}{U_q(\mathfrak{sl}_n)} 

\newcommand{\crosscurve}[5]{%
  \draw[#1, line width=1.1pt, postaction=decorate,
        decoration={markings, mark=at position #5 with {\arrow{Stealth[length=2.6mm]}}}]
    #2 to[bend left=#4] #3;}

\tikzset{midmark/.style={postaction=decorate,
  decoration={markings, mark=at position 0.5 with {\arrow{Stealth[length=2.6mm]}}}}}

\tikzset{contour/.style={line width=1.1pt, rounded corners=2.5pt, postaction=decorate,
  decoration={markings,
    mark=at position 0.50 with {\arrow{Stealth[length=2.6mm]}}}}}

\newcommand{\strandA}[1]{\draw[#1,contour]
  (1.328,-2.096)--(-0.072,-1.046)--(-0.12,-0.95)--(-0.12,0.95)--(-0.072,0.854)--(-1.472,1.904);}
\newcommand{\strandB}[1]{\draw[#1,contour]
  (1.472,1.904)--(0.072,0.854)--(0.12,0.95)--(0.12,-0.95)--(0.072,-1.046)--(-1.328,-2.096);}

\tikzset{>={Stealth[length=2.6mm]}}
\tikzset{axis/.style={dashed, line width=0.7pt}}
\colorlet{depthcolor}{brown!75!black}

\newtheorem{theorem}{Theorem}
\newtheorem{definition}[theorem]{Definition}
\newtheorem{lemma}[theorem]{Lemma}
\newtheorem{remark}[theorem]{Remark}
\newtheorem{example}[theorem]{Example}
\newtheorem{proposition}[theorem]{Proposition}
\newtheorem{corollary}[theorem]{Corollary}
\newtheorem{construction}[theorem]{Construction}

\title{Leading term strandings for webs}
\date{}

\author[Bo]{Michael Bo}
\address{University of Richmond, Jepson Hall, Richmond VA 23173}
\email{michael.bo@richmond.edu}

\author[Burns]{Madelyn Burns}
\address{University of Richmond, Jepson Hall, Richmond VA 23173}
\email{madelyn.burns@richmond.edu}

\author[Chen]{Junyang Chen}
\address{University of Richmond, Jepson Hall, Richmond VA 23173}
\email{junyang.chen@richmond.edu}

\author[Marsho]{Blaise Marsho}
\address{University of Richmond, Jepson Hall, Richmond VA 23173}
\email{blaise.marsho@richmond.edu}

\author[Martin]{Jacob Martin}
\address{University of Richmond, Jepson Hall, Richmond VA 23173}
\email{jacob.martin@richmond.edu}

\author[Mawn]{Jade Mawn}
\address{University of Richmond, Jepson Hall, Richmond VA 23173}
\email{jade.mawn@richmond.edu}

\author[Mohren]{Bella Mohren}
\address{Georgia Institute of Technology, 686 Cherry Street, Atlanta, GA 30332}
\email{imohren3@gatech.edu}

\author[Russell]{Heather M. Russell}
\address{University of Richmond, Jepson Hall, Richmond VA 23173}
\email{hrussell@richmond.edu}

\author[Sales]{Caitlin Sales}
\address{University of Richmond, Jepson Hall, Richmond VA 23173}
\email{caitlin.sales@richmond.edu}

\author[Wong]{Lance Wong}
\address{University of Richmond, Jepson Hall, Richmond VA 23173}
\email{lance.wong@richmond.edu}

\begin{document}

\begin{abstract}
A web is a plane graph encoding an invariant vector in a tensor product of
fundamental representations of a quantum group. A stranding of an
$\mathfrak{sl}_n$ web is a system of colored oriented curves recording one
monomial of the vector it encodes. This article focuses on identifying and
constructing \emph{leading term strandings}, those recording the leading term
of a web's vector with respect to a lexicographic order on monomials. We show
that every open strand of a leading term stranding is clockwise, so the boundary data of such a stranding is recorded by a row-strict tableau. From a row-strict tableau, we
construct a web with a prescribed leading term, and the resulting webs form a web basis. This non-recursive construction recovers Fontaine's bases and produces others. For $\mathfrak{sl}_3$ webs with no flat vertices,
we identify a leading term stranding using the depths of the faces of the web.
Finally, we ask whether any two strandings of a web are connected by a sequence of
local moves called strand reversals. We answer this question affirmatively for
$\mathfrak{sl}_3$, extending a known result for closed webs to webs with boundary; in particular, a leading term stranding of any $\mathfrak{sl}_3$
web is always reachable from an arbitrary stranding.
\end{abstract}

\maketitle

\section{Introduction}\label{section: introduction}
Webs are plane graphs that encode invariant vectors in tensor products of fundamental
representations of a quantum group; here, we focus on $\mathfrak{sl}_n$ webs. A
central and longstanding problem is to identify \emph{web bases}: sets of web graphs
whose vectors form a basis of an invariant space.
Since they encode algebraic
operations via manipulations of graphs, web bases are concrete and amenable to computation. Their combinatorics has found applications
ranging from tangle invariants \cite{KhovanovTangles} to cluster algebras \cite{FP16}.
Approaches to establishing web bases often hinge on triangularity, which requires one to
 isolate a leading term with respect to a chosen order. For bases of tensor
invariants, that leading term is a monomial in the expansion of the web vector. For the $\mathfrak{sl}_2$ and $\mathfrak{sl}_3$ web bases described below, there is a straightforward process for reading the leading term directly from the web graph \cite{BazierLeading}.
For a general web, no such process is known. 

With this in mind, the question driving
our work is:

\begin{tcolorbox}[boxrule=0.4pt, colback=white, colframe=black]
\begin{center}
    \emph{For an arbitrary web, how can we use its structural properties\\
    to infer the
    leading term of its web vector?}
\end{center}
\end{tcolorbox}
\noindent Every term of a web vector arises from a system of oriented curves drawn on the web. Our first result shows that for the leading term, every open curve in such a system runs clockwise, which constrains which monomials can be leading terms. We then approach leading terms in two ways: building webs with prescribed leading terms and finding leading terms of given webs.

Web bases are not unique, but in low rank, distinguished bases with strong
combinatorial structure have been identified. For
$\mathfrak{sl}_2$, these are given by  Temperley--Lieb diagrams
\cite{TemperleyLieb}, crossingless matchings enumerated by the Catalan numbers. For $\mathfrak{sl}_3$, Kuperberg identified the
basis of nonelliptic webs \cite{KuperbergWebs}. Khovanov and Kuperberg gave a recursive algorithm building the elements
of this basis from three-row tableaux \cite{KK}, and Tymoczko \cite{TSimpleBij} and
later Russell \cite{RussellSemistandard} gave an explicit
construction producing the same basis without appealing to recursion. These tableaux reflect the structure of the corresponding webs and support the study of group actions on webs \cite{PPR, RTshadow, RTTransitionsl2,  TSimpleBij}.

For $n\geq 4$ the available results are less complete. Westbury \cite{WestburyBases}
constructs web bases when each tensor factor is the standard representation or its dual,
and Fontaine \cite{FON} extends this to arbitrary tensor products of fundamental
representations; Elias \cite{Eli15} obtained closely related bases in the quantum
setting. Each of these builds its webs by a recursive process, in layers rather than in
a single step. The resulting bases are parameterized by tableaux, but we lack a simple
criterion for deciding whether a given web is a basis element. These bases also lack some of the combinatorial
features of the low-rank bases; in particular, they are not invariant under rotation of
the boundary \cite{GaetzetalRotation}. For $n=4$, Gaetz, Pechenik, Pfannerer, Striker,
and Swanson \cite{GaetzetalRotation} construct a rotation-invariant basis indexed by
fluctuating tableaux, again recursively, by way of a large system
of growth rules.
Many of the available tools are
tied to the settings in which they were developed. Fontaine works with coherent webs --- a restriction imposed by the geometric setting used to establish his
bases. The approach of \cite{GaetzetalRotation} is
carried out for fully reduced graphs, again a restricted class.

To establish other bases and carry out computations for general $n$, we need combinatorial tools to extract algebraic information from webs.
We approach this through \emph{strandings} \cite{RTStranding}, which were developed as a tool for working with arbitrary $\mathfrak{sl}_n$ webs. A stranding of a web $G$
is a system of colored oriented curves drawn on $G$. A single web supports many valid
strandings, and they collectively encode the web vector; each stranding contributes a single monomial, and these contributions never cancel. A stranding whose monomial
is the leading term of the web vector is a \emph{leading term stranding}.
A comparison of strandings with related constructions in the literature appears in
\cite{RTStranding}.

The notion of depth relates the complexity of a web to the strandings it supports. Each
face of a web carries an \emph{integral depth}, coming from distance in the web graph's dual, and
every stranding assigns to a face a \emph{strand depth}, counting the simple strands that
enclose it. Strand depth never exceeds integral depth (Lemma~\ref{lem: depth bound}). Since integral depth depends only on the web, this bound constrains which strandings a web supports, and it is our main tool for identifying leading term strandings. Integral and strand depth build on Tymoczko's
path and circle depth for $\mathfrak{sl}_3$ webs \cite{TSimpleBij}.

In Section~\ref{section: web vectors}, we prove that every open strand of a leading term
stranding is clockwise (Lemma~\ref{lem: open strands clockwise}). Consequently, the
tableau recording its boundary data is row-strict
(Theorem~\ref{thm: leading term row strict}), which gives a basis criterion that imposes
no structural restriction on the webs but requires identifying leading term strandings
(Corollary~\ref{cor: basis criterion via tableaux}). 

In Section~\ref{section: matchings}, we build a web with a prescribed leading term. From a
row-strict tableau we construct a system of colored arcs, a \emph{colored noncrossing
matching}, and then produce a web by local modifications at the crossings and boundary
vertices of that matching, with the arcs of the matching recovering a leading term stranding of the resulting web. The
construction is explicit rather than recursive, playing for general $n$ the role that
\cite{RussellSemistandard, TSimpleBij} play for $\mathfrak{sl}_3$, and the webs so
obtained form a web basis indexed by row-strict tableaux
(Corollary~\ref{cor: matching produces basis}). When $n=3$, the construction returns exactly the nonelliptic web basis of Kuperberg.
Starting at $n=4$, the matching arcs must be ordered at each boundary
vertex, a choice also present in Fontaine's algorithm, but our construction has
additional flexibility coming from isotopy of the matching arcs away from the boundary.
The webs produced by Fontaine's algorithm arise from our construction, and the extra flexibility produces others
(Example~\ref{ex: same leading term}).

In Section~\ref{section: sl3 stranding}, we seek the leading term of a given web. For
$\mathfrak{sl}_3$ webs with no \emph{flat} vertices, integral depth determines a unique
compatible stranding, and that stranding is a leading term stranding
(Theorem~\ref{thm:sl3 stranding}).
With respect to a specific order, Bazier-Matte, Douville, Garver, Patrias, Thomas, and
Y{\i}ld{\i}r{\i}m identify the minimal term in the vector associated to a nonelliptic
$\mathfrak{sl}_3$ web \cite{BazierLeading}.
Our class includes every nonelliptic web together
with some, though not all, elliptic ones, so a leading term becomes available for
a strictly larger class of webs than was previously accessible.
Such elliptic webs arise in the
parametrization of top-dimensional components of three-row Springer fibers \cite{HLTT}.
Flat vertices obstruct this argument, and for $n\geq 4$, integral depth is too coarse to
single out a stranding at all. 

As an alternative we consider reconfiguration by strand reversal, a local move on the
set of valid strandings. This suggests a natural question, of interest independent of
leading terms:

\begin{tcolorbox}[boxrule=0.4pt, colback=white, colframe=black]
\begin{center}
    \emph{Is there a sequence of strand reversals connecting\\
    any two strandings of a web?}
\end{center}
\end{tcolorbox}

\noindent We answer this affirmatively for $n=3$: strand reversal acts transitively on
the strandings of an $\mathfrak{sl}_3$ web (Corollary~\ref{cor: sl3 kempe}), so a
leading term stranding is always reachable from any starting stranding.
For closed webs, this result is established in several places in the literature
\cite{Fisk, belcastroHaas, Robert3col}. The extension to $\mathfrak{sl}_3$ webs with
boundary is the new contribution.
It remains open
to produce such a sequence algorithmically, and for $n\geq 4$ the question is open.

\section{Webs, strandings, and depth}\label{section: background}

In this article, we work with \emph{untagged} $\mathfrak{sl}_n$ webs, which are directed, edge-weighted plane
graphs satisfying a flow condition mod $n$ at each trivalent interior vertex, as in
\cite{FontaineThesis, FON, KuperbergWebs,  WestburyBases}. The tagged webs of \cite{CKM}
matter for categorical and sign-sensitive purposes, but tags complicate the explicit
combinatorial calculations carried out here.

The present section fixes our conventions for webs, collects the
definitions and results on strandings that we use from \cite{RTStranding}, and develops
the depth framework.

\subsection{Half-plane graphs and webs}\label{subsection: webs}

All the graphs in this paper are embedded in a half-plane with some of their
vertices on its boundary line.

\begin{definition}[Half-plane graphs]\label{def: half-plane graph}
Fix a horizontal line in $\mathbb{R}^2$, called the \textbf{boundary axis}. A
\textbf{half-plane graph} is a finite graph $G$ embedded in the closed half-plane
below the boundary axis, taken up to planar isotopy fixing that axis, meeting it
in a finite set of vertices.
\begin{itemize}
\item Vertices of $G$ on the boundary axis are \textbf{boundary vertices}; all
others are \textbf{interior vertices}. An edge incident to a boundary vertex is a
\textbf{boundary edge}; every other edge is an \textbf{interior edge}.
\item The \textbf{faces} of $G$ are the connected components of the closed
half-plane with $G$ removed. The unbounded face is denoted $U$. A face containing points of the boundary axis is a \textbf{boundary face}; the rest are \textbf{interior faces}.
\end{itemize}
\end{definition}

\begin{remark}
In \cite{RTshadow}, half-plane graphs are additionally required to have
trivalent interior vertices and univalent boundary vertices. We impose no condition on vertex degrees here,
because we apply the definition not only to webs but to the
noncrossing matchings of Section~\ref{section: matchings}, whose vertices may have higher degree.
\end{remark}

The primary example of a half-plane graph in this paper is a web: a half-plane
graph whose edges are directed and weighted, subject to a congruence condition
at each interior vertex. To state that condition, we first record the direction
of an edge relative to an incident vertex. For a vertex $v$ and an oriented edge
$e$ incident to $v$, write
$$
\sigma_v(e)=
\begin{cases}
1 & \textup{if $e$ is directed into $v$},\\
-1 & \textup{if $e$ is directed out of $v$}.
\end{cases}
$$

\begin{definition}[Webs]\label{def: webs}
Fix an integer $n\geq 2$.
\begin{itemize}
    \item An {\bf $\mathfrak{sl}_n$ web} (or simply {\bf web} when $n$ is clear from context) is a directed half-plane graph $G$ with edge weights in $\{1,\ldots,n-1\}$. Boundary vertices are univalent, and interior vertices are trivalent. At each interior vertex $v$ with incident edges $e_1,e_2,e_3$ of weights $\ell_1,\ell_2,\ell_3$, an $\mathfrak{sl}_n$ web satisfies the following congruence:
$$
\sum_{i=1}^3 \sigma_v(e_i)\ell_i \equiv 0 \pmod n.
$$
\item Given a web with boundary vertices $v_1,\ldots,v_m$ read from left to right and incident boundary edges $e_i$ of weights $\ell_i$, its {\bf boundary weight vector} is $\vec{k}=(k_1,\ldots,k_m)$, where
$$
k_i=
\begin{cases}
\ell_i & \textup{if } \sigma_{v_i}(e_i)=1,\\
n-\ell_i & \textup{if } \sigma_{v_i}(e_i)=-1.
\end{cases}
$$

\item Let $F_n(\vec{k})$ denote the set of $\mathfrak{sl}_n$ webs with boundary weight vector $\vec{k}$.
\end{itemize}
\end{definition}

\begin{example}\label{ex: running web}
Figure~\ref{fig: running web} shows a web $G\in F_4(1,1,1,1,1,1,1,1)$. Single edges have weight $1$, and double edges have weight 2. We use this web throughout Section~\ref{section: background}  to illustrate various aspects of the stranding setup.
\begin{figure}[h]
\centering
\begin{tikzpicture}[scale=1, yscale=-1]
\draw[axis, <->] (.5,.25)--(8.5,.25);

\begin{scope}[line width=0.7pt,decoration={markings,mark=at position 0.5 with {\arrow{<}}}]
\draw[postaction={decorate}] (1,.25)--(2,1);
\draw[postaction={decorate}] (2,.25)--(2,1);
\draw[postaction={decorate}] (3,.25)--(3.25,1);
\draw[postaction={decorate}] (4,.25)--(3.75,1);
\draw[postaction={decorate}] (3.25,2)--(3.25,1);
\draw[postaction={decorate}] (3.25,2)--(4.5,2.5);
\draw[postaction={decorate}] (8,.25) to[out=120, in=0] (4.5,2.5);
\draw[postaction={decorate}] (4.5,1.5)--(3.75,1);
\draw[postaction={decorate}] (4.5,1.5)--(5,1);
\draw[postaction={decorate}] (5,.25)--(5,1);
\draw[postaction={decorate}] (6,.25)--(6.5,1);
\draw[postaction={decorate}] (7,.25)--(6.5,1);
\end{scope}

\begin{scope}[line width=0.7pt,decoration={markings,mark=at position 0.5 with {\arrow{<}}}]
\draw[double, postaction={decorate}] (2,1)--(3.25,2);
\draw[double, postaction={decorate}] (3.25,1)--(3.75,1);
\draw[double, postaction={decorate}] (5,1)--(6.5,1);
\draw[double, postaction={decorate}] (4.5,1.5)--(4.5,2.5);
\end{scope}

\foreach \x in {1,...,8}
  \fill(\x,.25) circle (2.5pt);
\fill(2,1) circle (2.5pt);
\fill(3.25,1) circle (2.5pt);
\fill(3.75,1) circle (2.5pt);
\fill(5,1) circle (2.5pt);
\fill(6.5,1) circle (2.5pt);
\fill(4.5,1.5) circle (2.5pt);
\fill(3.25,2) circle (2.5pt);
\fill(4.5,2.5) circle (2.5pt);
\end{tikzpicture}
\caption{An $\mathfrak{sl}_4$ web.}\label{fig: running web}
\end{figure}
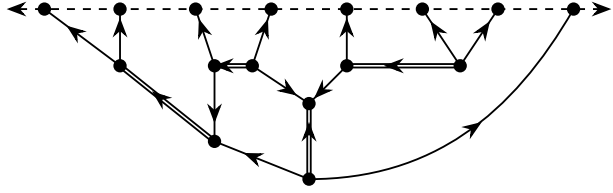
\end{example}

\subsection{Curve systems, strandings, and binary labelings}\label{subsection: strandings and binary labelings}

We are interested in networks of oriented curves that run along the edges of the half-plane graphs described above.

\begin{definition}[Arcs, loops, and curve systems]\label{def: arcs and loops}
Let $G$ be a half-plane graph. 
\begin{itemize}
    \item An \textbf{arc on $G$} is an oriented, simple path in $G$ with both endpoints on the boundary axis. An arc is clockwise (resp. counterclockwise) if its terminal vertex is to the left (resp. right) of its initial vertex. 
    \item A \textbf{loop on $G$} is an oriented, simple closed path in $G$; it is clockwise or counterclockwise according to its orientation. 
    \item A \textbf{curve system} on $G$ is a collection of arcs and loops on $G$. A curve of the system is \textbf{open} if it is an arc and \textbf{closed} if it is a loop.
    \item A curve \textbf{encloses} a face $F$ of $G$ if it separates $F$ from $U$ in the half-plane. 
    \end{itemize}
    If $G$ is directed, the orientations of arcs and loops on $G$ need not agree with the orientations of the edges they traverse. 
\end{definition}

The primary example of a curve system in this paper is a stranding of a web, in which the arcs and loops are colored so that those of the same color are pairwise disjoint, subject to local conditions along each edge. The next two
definitions record the same data in different ways -- globally, as a curve
system on $G$, and locally, as a binary vector on each edge. Our notation
differs slightly from \cite[Section 4]{RTStranding}.

\begin{definition}[Strandings]\label{def: stranding}
Let $G\in F_n(\vec{k})$.
\begin{itemize}
\item A {\bf stranding} $S$ of $G$ is a curve system on $G$ whose curves, called {\bf simple strands}, are colored by $\{1,\ldots,n-1\}$ such that simple strands of the same color are pairwise disjoint. 

\item For an edge $e$ of $G$, write
$$
S(e)=S^+(e)\sqcup S^-(e),
$$
where $S^+(e)$ (resp. $S^-(e)$) is the set of colors of simple strands directed with (resp. against) the orientation of $e$. Define
$$
\alpha_c^\vee(S(e))=
\begin{cases}
1 & \textup{if } c\in S^+(e),\\
-1 & \textup{if } c\in S^-(e),\\
0 & \textup{if } c\notin S(e).
\end{cases}
$$
\item A stranding is {\bf valid} if for each edge $e$ of $G$ the following two conditions are met.
\begin{enumerate}
    \item The nonzero values $\alpha_c^\vee(S(e))$ alternate in sign as $c$ increases.
    \item For the largest color $c_{\max}$ in $S(e)$, we have
    $$
    \sum_{c=1}^{n-1}\alpha_c^\vee(S(e))c=
    \begin{cases}
        \ell & \textup{if } c_{\max}\in S^+(e),\\
        \ell-n & \textup{if } c_{\max}\in S^-(e).
    \end{cases}
    $$
\end{enumerate}
\end{itemize}
Write $\mathcal{S}tr(G)$ for the set of valid strandings of $G$.
 \end{definition}

Next, we observe that for any set of simple strands satisfying validity condition (1), we can choose an edge orientation so that the
alternating sum in condition (2) lies in
$\{1,\ldots,n-1\}$. This is relevant to the construction in
Section~\ref{section: matchings} where we first choose a stranding and then build a web that is compatible with it.

\begin{lemma}\label{lem: alternating sum weight}
Let $S$ be an alternating-sign collection of colored simple strands between vertices $u$ and $v$, with the largest color directed into $v$. Let $e:=u\mapsto v$, weighted by $\ell:=\sum_{c=1}^{n-1}\alpha_c^\vee(S(e))c$. Then $\ell\in\{1,\ldots,n-1\}$, and $S(e)$ is a valid stranding of $e$.
\end{lemma}

\begin{proof}
List the colors of $S(e)$ in increasing order as $0<c_1<c_2<\cdots<c_k<n$. Since the nonzero values $\alpha_c^\vee(S(e))$ alternate in sign and the largest color $c_k$ is directed with $e$, we have $\alpha_{c_i}^\vee(S(e))=(-1)^{i+k}$, so that
$$
\ell=\sum_{i=1}^k (-1)^{i+k}c_i.
$$
Set $c_0:=0$. Pairing consecutive terms from the top,
$\ell=\sum_{j=0}^{\lceil k/2\rceil -1}(c_{k-2j}-c_{k-2j-1})$, a sum of strictly positive terms, so $\ell>0$. These terms are among the terms of the telescoping sum $(c_k-c_{k-1})+(c_{k-1}-c_{k-2})+\cdots+(c_1-c_0)=c_k$, and every omitted term is positive. Hence $0<\ell\leq c_k<n$ and $\ell\in\{1,\ldots,n-1\}$.
\end{proof}

\begin{definition}[Binary labelings]\label{def: binary labeling}
Let $G\in F_n(\vec{k})$. 
\begin{itemize}
    \item A {\bf binary labeling} of $G$ is an assignment of a binary vector $b(e)=b_1\cdots b_n\in\{0,1\}^n$ to each edge $e$ such that $|b(e)|:= b_1 + \cdots + b_n = \ell$ if $e$ has weight $\ell$. 

    \item For $1\leq c<n$, define
$$
\alpha_c^\vee(b(e))=b_{c}-b_{c+1}.
$$ 
\item We call a binary labeling {\bf valid} if, at each interior vertex $v$ with incident edges $e_1,e_2,e_3$ and each $1\leq c< n$, we have
$$
\sum_{i=1}^3 \sigma_v(e_i)\,\alpha_c^{\vee}(b(e_i))=0.
$$
In other words, $\sum_{i=1}^3 \sigma_v(e_i)\,b(e_i)_j$ is constant across all $1\leq j\leq n$. 
\end{itemize}
Write $\mathcal{B}in(G)$ for the set of valid binary labelings of $G$.
 \end{definition}

The following technical lemma follows from the definition of a valid binary labeling. It will be useful in Section~\ref{section: matchings}.

\begin{lemma}\label{lem: flow condition on binary labels}
Let $G\in F_n(\vec{k})$, and let $v$ be an interior vertex with incident edges $e_1, e_2, e_3$ of weights $\ell_1, \ell_2, \ell_3$ such that $\sum_{i=1}^3 \sigma_v(e_i)\ell_i=0$. Consider $b\in \mathcal{B}in(G)$. Then $\sum_{i=1}^3 \sigma_v(e_i)b(e_i)=\vec{0}$. In other words, at such a vertex the binary label of an edge is a $\{\pm 1\}$-linear combination of the binary labels of the other two edges.
\end{lemma}

\begin{proof}
Set $w=\sum_{i=1}^3 \sigma_v(e_i)b(e_i)$. Validity of $b$ says
$\alpha_c^{\vee}(w)=w_c-w_{c+1}=0$ for all $1\leq c<n$, so all entries of $w$ are
equal. Summing the entries of $b(e_i)$ gives $\ell_i$, so the entries of $w$ sum to
$\sum_{i=1}^3\sigma_v(e_i)\ell_i=0$. Hence $w=\vec 0$.
\end{proof}

It is convenient to express binary labels in terms of the fundamental-weight
binary vectors. For $1\leq c\leq n-1$, let
$$
\lambda_c := \underbrace{1\cdots 1}_{c}\underbrace{0\cdots 0}_{n-c}\;\in\;\{0,1\}^n,
$$
the binary representative of the $c$th fundamental weight of $\mathfrak{sl}_n$. Stranding and binary labeling are equivalent via the bijection described below.

\begin{theorem}[{\cite[Theorem 37]{RTStranding}}]\label{thm:bijection} Given a web $G\in F_n(\vec{k})$, the following mutually inverse maps define a bijection between $\mathcal{B}in(G)$ and $\mathcal{S}tr(G)$. 

\begin{itemize}
\item Let $b\in \mathcal{B}in(G)$. For each edge $e$ in $G$, define the stranding $S_b$ by

$$S_b^+(e)=\{c:\alpha_c^{\vee}(b(e))=1\} \text{ and } S_b^-(e)=\{c:\alpha_c^{\vee}(b(e))=-1\}.$$ 

\item Let $S\in\mathcal{S}tr(G)$. For each edge $e$ in $G$ where $c_{\max}$ is the largest value in $S(e)$, define the binary labeling
$$
b_S(e)=
\begin{cases}
\displaystyle\sum_{c=1}^{n-1}\alpha_c^\vee(S(e))\,\lambda_c 
   & \textup{if } c_{\max}\in S^+(e),\\[1.2em]
\vec{1}+\displaystyle\sum_{c=1}^{n-1}\alpha_c^\vee(S(e))\,\lambda_c 
   & \textup{if } c_{\max}\in S^-(e).
\end{cases}
$$
\end{itemize}

\end{theorem}
Unwinding the second map, $b_S(e)$ is the binary vector whose last entry is $0$
if $c_{\max}\in S^+(e)$ and $1$ if $c_{\max}\in S^-(e)$, and which satisfies
$\alpha_c^\vee(b_S(e))=\alpha_c^\vee(S(e))$ for all $1\leq c\leq n-1$.

A valid stranding induces a larger collection of curves on $G$, indexed by pairs $1\leq i<j\leq n$.

\begin{definition}[$(i,j)$-Strands]\label{def: ij-strands}
Let $G\in F_n(\vec{k})$ and $S\in\mathcal{S}tr(G)$.
\begin{itemize}
\item For each edge $e$ of $G$ with $b_S(e)=b_1\cdots b_n$, define
$$
L(e)=\left\{(i,j): 1\le i<j\le n,\ b_i\neq b_j\right\}=L^+(e)\sqcup L^-(e),
$$
where
$$
L^\pm(e)=\left\{(i,j): 1\le i<j\le n,\ b_i-b_j=\pm 1\right\}.
$$

\item For each $1\leq i<j\leq n$, let $L_{(i,j)}(S)$ be the curve system on $G$ that traverses each edge $e=u\mapsto v$ of $G$
\begin{itemize}
\item from $u$ to $v$ if $(i,j)\in L^+(e)$,
\item from $v$ to $u$ if $(i,j)\in L^-(e)$,
\item and not at all if $(i,j)\notin L(e)$.
\end{itemize}
These assemble into arcs and loops \cite[Lemma~25]{RTStranding}, and we call each connected component of $L_{(i,j)}(S)$ an {\bf $(i,j)$-strand} of $S$.
\end{itemize}
\end{definition}

Since $\alpha_c^\vee(b_S(e))=b_c-b_{c+1}$, the simple strands of color $c$ in $S$ are exactly the $(c,c+1)$-strands.

\begin{example}\label{ex: ij-strands}
On the left in Figure~\ref{fig: running stranding}, we have an example of a stranding $S$ of the web from Figure~\ref{fig: running web} together with the corresponding binary labeling $b_S$. We draw simple strands of colors $1,2,3$ in blue, red, and green, respectively. On the right, we illustrate the collection of $(2,4)$-strands in orange.

Consider the interior vertical edge $e$ of weight $\ell=2$ in the center of the web which is stranded with blue and green simple strands directed with $e$ and a red simple strand directed against $e$. This implies $S^{+}(e)=\{1,3\}$ and $S^{-}(e)=\{2\}$. Along this edge, $c_{\max}=3\in S^+(e)$, so Theorem~\ref{thm:bijection} gives
$$b_S(e) = \sum_{c=1}^{3}\alpha_c^\vee(S(e))\,\lambda_c 
= 1\cdot\lambda_1 + (-1)\cdot\lambda_2 + 1\cdot\lambda_3 
= 1010.$$
Since $b_2-b_4=0-0=0$, this edge does not carry a $(2,4)$ strand.

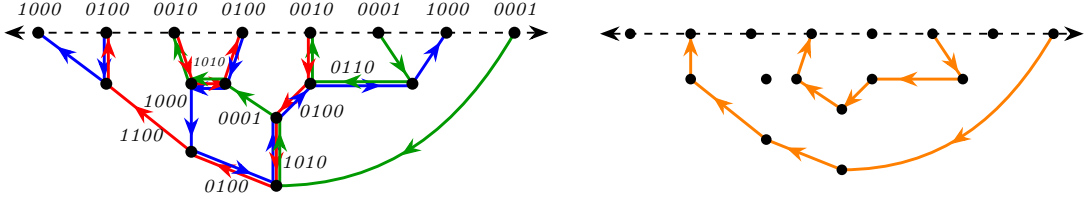
\begin{figure}[h]
\centering
\begin{tikzpicture}[scale=.9, yscale=-1]
\draw[axis, <->] (.5,.25)--(8.5,.25);

\begin{scope}[line width=1.1pt,decoration={markings,mark=at position 0.5 with {\arrow{<}}}]
\draw[blue,postaction={decorate}] (1,.25)--(2,1);
\draw[blue,postaction={decorate}] (1.97,1)--(1.97,.25);
\draw[red,postaction={decorate}] (2.03,.25)--(2.03,1);
\draw[green!60!black, postaction={decorate}] (2.97,.25)--(3.22,1);
\draw[red, postaction={decorate}] (3.28,1)--(3.03,.25);
\draw[red, postaction={decorate}] (3.97,.25)--(3.72,1);
\draw[blue, postaction={decorate}] (3.78,1)--(4.03,.25);
\draw[blue, postaction={decorate}] (3.25,2)--(3.25,1);
\draw[red, postaction={decorate}] (3.25,2.04)--(4.5,2.54);
\draw[blue, postaction={decorate}] (4.5,2.47)--(3.25,1.97);
\draw[green!60!black, postaction={decorate}] (4.5,2.5) to[out=0, in=120] (8,.25);
\draw[green!60!black, postaction={decorate}] (3.75,1)--(4.5, 1.5);
\draw[red,postaction={decorate}] (4.5, 1.45)--(5,.95);
\draw[blue,postaction={decorate}] (5, 1.05)--(4.5,1.55);
\draw[green!60!black, postaction={decorate}] (5.03, .25)--(5.03,1);
\draw[red, postaction={decorate}] (4.97, 1)--(4.97,.25);
\draw[green!60!black, postaction={decorate}] (6.5, 1)--(6,.25);
\draw[blue, postaction={decorate}] (7, .25)--(6.5,1);
\end{scope}

\begin{scope}[decoration={markings,mark=at position 0.5 with {\arrow{<}}}]
\draw[red, line width=1.1pt, postaction={decorate}] (2,1)--(3.25,2);
\draw[blue, line width=1.1pt, postaction={decorate}] (3.25,1.07)--(3.75,1.07);
\draw[red, line width=1.1pt, postaction={decorate}] (3.75,1)--(3.25,1);
\draw[green!60!black, line width=1.1pt, postaction={decorate}] (3.25,.93)--(3.75,.93);
\draw[blue, line width=1.1pt, postaction={decorate}] (6.5,1.03)--(5,1.03);
\draw[green!60!black, line width=1.1pt, postaction={decorate}] (5,.97)--(6.5,.97);
\draw[blue, line width=1.1pt, postaction={decorate}] (4.45,1.5)--(4.45,2.5);
\draw[red, line width=1.1pt, postaction={decorate}] (4.5,2.5)--(4.5,1.5);
\draw[green!60!black, line width=1.1pt, postaction={decorate}] (4.55,1.5)--(4.55,2.5);
\end{scope}

\foreach \x in {1,...,8}
  \fill(\x,.25) circle (2.5pt);
\fill(2,1) circle (2.5pt);
\fill(3.25,1) circle (2.5pt);
\fill(3.75,1) circle (2.5pt);
\fill(5,1) circle (2.5pt);
\fill(6.5,1) circle (2.5pt);
\fill(4.5,1.5) circle (2.5pt);
\fill(3.25,2) circle (2.5pt);
\fill(4.5,2.5) circle (2.5pt);

\node at (1,-.1) {\tiny{1000}};
\node at (2,-.1) {\tiny{0100}};
\node at (3,-.1) {\tiny{0010}};
\node at (4,-.1) {\tiny{0100}};
\node at (5,-.1) {\tiny{0010}};
\node at (6,-.1) {\tiny{0001}};
\node at (7,-.1) {\tiny{1000}};
\node at (8,-.1) {\tiny{0001}};
\node at (2.5,1.75) {\tiny{1100}};
\node at (2.85,1.25) {\tiny{1000}};
\node at (3.5,.675) {\fontsize{5}{6}\selectfont 1010};
\node at (3.95,1.5) {\tiny{0001}};
\node at (3.75,2.5) {\tiny{0100}};
\node at (5.15,1.4) {\tiny{0100}};
\node at (5.6,.75) {\tiny{0110}};
\node at (4.9,2.15) {\tiny{1010}};
\end{tikzpicture}
\hspace{.15in}
\raisebox{10pt}{\begin{tikzpicture}[scale=.8, yscale=-1]
\draw[axis, <->] (.5,.25)--(8.5,.25);

\begin{scope}[line width=1.1pt,decoration={markings,mark=at position 0.5 with {\arrow{<}}}]

 \draw[orange,postaction={decorate}] (2,.25)--(2,1);

 \draw[orange, postaction={decorate}] (4,.25)--(3.75,1);

\draw[orange, postaction={decorate}] (3.25,2)--(4.5,2.5);

 \draw[orange, postaction={decorate}] (4.5,2.5) to[out=0, in=120] (8,.25);
\draw[orange, postaction={decorate}] (3.75,1)--(4.5, 1.5);
 \draw[orange,postaction={decorate}] (4.5, 1.5)--(5,1);
\draw[orange, postaction={decorate}] (6.5, 1)--(6,.25);

\end{scope}

\begin{scope}[decoration={markings,mark=at position 0.5 with {\arrow{<}}}]
 \draw[orange, line width=1.1pt, postaction={decorate}] (2,1)--(3.25,2);
 \draw[orange, line width=1.1pt, postaction={decorate}] (5,1)--(6.5,1);
\end{scope}

\foreach \x in {1,...,8}
  \fill(\x,.25) circle (2.5pt);
\fill(2,1) circle (2.5pt);
\fill(3.25,1) circle (2.5pt);
\fill(3.75,1) circle (2.5pt);
\fill(5,1) circle (2.5pt);
\fill(6.5,1) circle (2.5pt);
\fill(4.5,1.5) circle (2.5pt);
\fill(3.25,2) circle (2.5pt);
\fill(4.5,2.5) circle (2.5pt);
\end{tikzpicture}}
\caption{A stranding $S$ together with its binary labeling $b_S$ (left) and the $(2,4)$ strands for $S$ (right).}\label{fig: running stranding}
\end{figure}
\end{example}

Strandings of a fixed web can be modified locally by reversing a single strand, giving a reconfiguration process on $\mathcal{S}tr(G)$. This operation was proposed by Emily Hafken and communicated to the authors by Julianna Tymoczko.

\begin{definition}[Strand reversal]\label{def: strand reversal}
Let $\gamma$ be an $(i,j)$-strand in a stranding $S$ of $G\in F_n(\vec{k})$, with
$1\leq i<j\leq n$. {\bf Reversing} $\gamma$ produces the stranding $S'$ whose binary
labeling $b_{S'}$ agrees with $b_S$ on every edge not in $\gamma$ and swaps the $i$th
and $j$th entries of $b_S(e)$ on every edge $e$ of $\gamma$, so that $\gamma$ is
oriented in the opposite direction in $S'$.
\end{definition}

\begin{remark}
    Robert's $\tau$-change operation \cite[Definition 2.1]{Robert3col} is a special case of strand reversal defined for closed $\mathfrak{sl}_3$ webs (see Corollary~\ref{cor: sl3 kempe}). 
\end{remark}

The next lemma shows that strand reversal always yields a valid stranding.

\begin{lemma}\label{lem: strand reversal}
Let $G\in F_n(\vec{k})$, $S\in\mathcal{S}tr(G)$, and let $\gamma$ be an $(i,j)$-strand of $S$ with $1\leq i<j\leq n$. Let $S'$ be the stranding obtained by reversing $\gamma$. Then $S'\in\mathcal{S}tr(G)$.
\end{lemma}

\begin{proof}
By Definition~\ref{def: ij-strands}, each edge $e\in\gamma$ has $b_S(e)_i\neq b_S(e)_j$, so exactly one of $b_S(e)_i$ and $b_S(e)_j$ is nonzero and $|b_{S'}(e)|=|b_S(e)|$ for all edges $e$. Therefore $b_{S'}$ is a binary labeling of $G$.

We verify validity at each interior vertex $v$ with incident edges $e_1,e_2,e_3$. If none of these edges is part of $\gamma$ then $b_{S'}$ agrees with $b_S$ at $v$ and there is nothing to show. Otherwise exactly two edges lie on $\gamma$; say $\gamma$ enters $v$ along $e_1$ and exits along $e_2$, so that
$$
\sigma_v(e_1)(b_S(e_1)_i - b_S(e_1)_j) = 1
\qquad\text{and}\qquad
\sigma_v(e_2)(b_S(e_2)_i - b_S(e_2)_j) = -1.
$$

Together with the definition of $b_{S'}$, these give
$$
\sigma_v(e_1)b_{S'}(e_1) + \sigma_v(e_2)b_{S'}(e_2)
= \sigma_v(e_1)b_S(e_1) + \sigma_v(e_2)b_S(e_2),
$$
so $\sum_{t=1}^3\sigma_v(e_t)\alpha_c^{\vee}(b_{S'}(e_t))
= \sum_{t=1}^3\sigma_v(e_t)\alpha_c^{\vee}(b_S(e_t)) = 0$ for all
$1\leq c < n$. Hence $b_{S'}$ is valid and $S'\in\mathcal{S}tr(G)$.
\end{proof}

\begin{example}\label{ex: strand reversal}
Figure~\ref{fig: strand reversal example} illustrates
Lemma~\ref{lem: strand reversal} on the stranding $S$ from
Figure~\ref{fig: running stranding}. The counterclockwise open color-$2$ simple strand with endpoints
$v_3$ and $v_4$ is outlined with a dashed curve. Reversing this strand produces the stranding $S'$ on the
right. 

\begin{center}
\begin{tikzpicture}[scale=.9, yscale=-1]
\draw[black, line width=0.6pt, dash pattern=on 2.5pt off 1.5pt, rounded corners=4pt]
  (2.80,.10) -- (3.09,1.18) -- (3.91,1.18) -- (4.20,.10) -- (3.82,.10) -- (3.62,.84) -- (3.38,.84) -- (3.18,.10) -- cycle;
\draw[axis, <->] (.5,.25)--(8.5,.25);

\begin{scope}[line width=1.1pt,decoration={markings,mark=at position 0.5 with {\arrow{<}}}]
\draw[blue,postaction={decorate}] (1,.25)--(2,1);
\draw[blue,postaction={decorate}] (1.97,1)--(1.97,.25);
\draw[red,postaction={decorate}] (2.03,.25)--(2.03,1);
\draw[green!60!black, postaction={decorate}] (2.97,.25)--(3.22,1);
\draw[red, postaction={decorate}] (3.28,1)--(3.03,.25);
\draw[red, postaction={decorate}] (3.97,.25)--(3.72,1);
\draw[blue, postaction={decorate}] (3.78,1)--(4.03,.25);
\draw[blue, postaction={decorate}] (3.25,2)--(3.25,1);
\draw[red, postaction={decorate}] (3.25,2.04)--(4.5,2.54);
\draw[blue, postaction={decorate}] (4.5,2.47)--(3.25,1.97);
\draw[green!60!black, postaction={decorate}] (4.5,2.5) to[out=0, in=120] (8,.25);
\draw[green!60!black, postaction={decorate}] (3.75,1)--(4.5, 1.5);
\draw[red,postaction={decorate}] (4.5, 1.45)--(5,.95);
\draw[blue,postaction={decorate}] (5, 1.05)--(4.5,1.55);
\draw[green!60!black, postaction={decorate}] (5.03, .25)--(5.03,1);
\draw[red, postaction={decorate}] (4.97, 1)--(4.97,.25);
\draw[green!60!black, postaction={decorate}] (6.5, 1)--(6,.25);
\draw[blue, postaction={decorate}] (7, .25)--(6.5,1);
\end{scope}

\begin{scope}[decoration={markings,mark=at position 0.5 with {\arrow{<}}}]
\draw[red, line width=1.1pt, postaction={decorate}] (2,1)--(3.25,2);
\draw[blue, line width=1.1pt, postaction={decorate}] (3.25,1.07)--(3.75,1.07);
\draw[red, line width=1.1pt, postaction={decorate}] (3.75,1)--(3.25,1);
\draw[green!60!black, line width=1.1pt, postaction={decorate}] (3.25,.93)--(3.75,.93);
\draw[blue, line width=1.1pt, postaction={decorate}] (6.5,1.03)--(5,1.03);
\draw[green!60!black, line width=1.1pt, postaction={decorate}] (5,.97)--(6.5,.97);
\draw[blue, line width=1.1pt, postaction={decorate}] (4.45,1.5)--(4.45,2.5);
\draw[red, line width=1.1pt, postaction={decorate}] (4.5,2.5)--(4.5,1.5);
\draw[green!60!black, line width=1.1pt, postaction={decorate}] (4.55,1.5)--(4.55,2.5);
\end{scope}

\foreach \x in {1,...,8}
  \fill(\x,.25) circle (2.5pt);
\fill(2,1) circle (2.5pt);
\fill(3.25,1) circle (2.5pt);
\fill(3.75,1) circle (2.5pt);
\fill(5,1) circle (2.5pt);
\fill(6.5,1) circle (2.5pt);
\fill(4.5,1.5) circle (2.5pt);
\fill(3.25,2) circle (2.5pt);
\fill(4.5,2.5) circle (2.5pt);

\node at (1,-.1) {\tiny{1000}};
\node at (2,-.1) {\tiny{0100}};
\node at (3,-.1) {\tiny{0010}};
\node at (4,-.1) {\tiny{0100}};
\node at (5,-.1) {\tiny{0010}};
\node at (6,-.1) {\tiny{0001}};
\node at (7,-.1) {\tiny{1000}};
\node at (8,-.1) {\tiny{0001}};
\node at (2.5,1.75) {\tiny{1100}};
\node at (2.9,1.25) {\tiny{1000}};
\node at (3.5,.75) {\tiny{1010}};
\node at (4,1.5) {\tiny{0001}};
\node at (3.75,2.5) {\tiny{0100}};
\node at (5.1,1.4) {\tiny{0100}};
\node at (5.6,.75) {\tiny{0110}};
\node at (4.85,2.15) {\tiny{1010}};
\end{tikzpicture}
\hspace{.15in}
\begin{tikzpicture}[scale=.9, yscale=-1]
\draw[axis, <->] (.5,.25)--(8.5,.25);

\begin{scope}[line width=1.1pt,decoration={markings,mark=at position 0.5 with {\arrow{<}}}]
\draw[blue,postaction={decorate}] (1,.25)--(2,1);
\draw[blue,postaction={decorate}] (1.97,1)--(1.97,.25);
\draw[red,postaction={decorate}] (2.03,.25)--(2.03,1);
\draw[red, postaction={decorate}] (2.97,.25)--(3.22,1);
\draw[blue, postaction={decorate}] (3.28,1)--(3.03,.25);
\draw[green!60!black, postaction={decorate}] (3.97,.25)--(3.72,1);
\draw[red, postaction={decorate}] (3.78,1)--(4.03,.25);
\draw[blue, postaction={decorate}] (3.25,2)--(3.25,1);
\draw[red, postaction={decorate}] (3.25,2.04)--(4.5,2.54);
\draw[blue, postaction={decorate}] (4.5,2.47)--(3.25,1.97);
\draw[green!60!black, postaction={decorate}] (4.5,2.5) to[out=0, in=120] (8,.25);
\draw[green!60!black, postaction={decorate}] (3.75,1)--(4.5, 1.5);
\draw[red,postaction={decorate}] (4.5, 1.45)--(5,.95);
\draw[blue,postaction={decorate}] (5, 1.05)--(4.5,1.55);
\draw[green!60!black, postaction={decorate}] (5.03, .25)--(5.03,1);
\draw[red, postaction={decorate}] (4.97, 1)--(4.97,.25);
\draw[green!60!black, postaction={decorate}] (6.5, 1)--(6,.25);
\draw[blue, postaction={decorate}] (7, .25)--(6.5,1);
\end{scope}

\begin{scope}[decoration={markings,mark=at position 0.5 with {\arrow{<}}}]
\draw[red, line width=1.1pt, postaction={decorate}] (2,1)--(3.25,2);
\draw[red, line width=1.1pt, postaction={decorate}] (3.25,1)--(3.75,1);
\draw[blue, line width=1.1pt, postaction={decorate}] (6.5,1.03)--(5,1.03);
\draw[green!60!black, line width=1.1pt, postaction={decorate}] (5,.97)--(6.5,.97);
\draw[blue, line width=1.1pt, postaction={decorate}] (4.45,1.5)--(4.45,2.5);
\draw[red, line width=1.1pt, postaction={decorate}] (4.5,2.5)--(4.5,1.5);
\draw[green!60!black, line width=1.1pt, postaction={decorate}] (4.55,1.5)--(4.55,2.5);
\end{scope}

\foreach \x in {1,...,8}
  \fill(\x,.25) circle (2.5pt);
\fill(2,1) circle (2.5pt);
\fill(3.25,1) circle (2.5pt);
\fill(3.75,1) circle (2.5pt);
\fill(5,1) circle (2.5pt);
\fill(6.5,1) circle (2.5pt);
\fill(4.5,1.5) circle (2.5pt);
\fill(3.25,2) circle (2.5pt);
\fill(4.5,2.5) circle (2.5pt);

\node at (1,-.1) {\tiny{1000}};
\node at (2,-.1) {\tiny{0100}};
\node at (3,-.1) {\tiny{0100}};
\node at (4,-.1) {\tiny{0010}};
\node at (5,-.1) {\tiny{0010}};
\node at (6,-.1) {\tiny{0001}};
\node at (7,-.1) {\tiny{1000}};
\node at (8,-.1) {\tiny{0001}};
\node at (2.5,1.75) {\tiny{1100}};
\node at (2.9,1.25) {\tiny{1000}};
\node at (3.5,.75) {\tiny{1100}};
\node at (4,1.5) {\tiny{0001}};
\node at (3.75,2.5) {\tiny{0100}};
\node at (5.1,1.4) {\tiny{0100}};
\node at (5.6,.75) {\tiny{0110}};
\node at (4.85,2.15) {\tiny{1010}};
\end{tikzpicture}

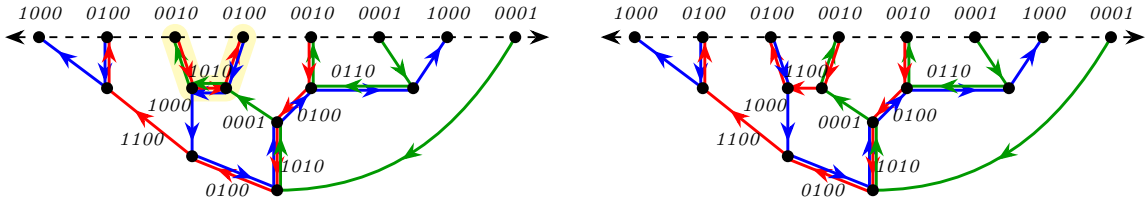
\captionof{figure}{Reversing the outlined counterclockwise open strand
of stranding $S$ (left) to obtain $S'$ (right).}
\label{fig: strand reversal example}
\end{center}
\end{example}

\subsection{Depth}\label{subsection: depth}

In this subsection we introduce two notions of depth for half-plane graphs and study how they interact. The first, integral depth, is a purely combinatorial measure of distance from the boundary. The second, depth with respect to a curve system, measures how many oriented curves enclose a given face. One of our main applications is to webs and their strandings, where curve depth specializes to a notion we call strand depth; later, we also use this same framework for the colored noncrossing matchings of Section~\ref{section: matchings}. 

\begin{definition}[Dual graph and integral depth]\label{def: dual graph}
Let $G$ be a half-plane graph.
\begin{itemize}
    \item The \textbf{dual graph} $G^*$ is the undirected graph with one vertex for each face of $G$ and one edge $e^*$ for each edge $e$ of $G$, connecting the two faces on either side of $e$.
    \item We describe a \textbf{path from $U$ to $F$ in $G^*$} as a sequence of faces $U=F_0,F_1,\ldots,F_d=F$ of $G$ in which consecutive faces $F_{i-1},F_i$ share an edge $e_i$ of $G$.
    \item For a face $F$ of $G$, the \textbf{integral depth} of $F$, denoted $d(F)$, is the length of a shortest path from $U$ to $F$ in $G^*$.
\end{itemize}
\end{definition}

Given a curve system on a half-plane graph, we define curve depth of a face as follows.

\begin{definition}[Curve depth and strand depth]\label{def: curve depth}
Let $G$ be a half-plane graph, and let $F$ be a face of $G$.
\begin{itemize}
\item Given a curve system $C$ on $G$, the \textbf{depth of $F$ with respect to $C$}, denoted $d_{C}(F)$, is the number of clockwise curves of $C$ enclosing $F$ minus the number of counterclockwise curves of $C$ enclosing $F$.
\item If $G$ is a web and $S\in \mathcal{S}tr(G)$, we call the curve depth coming from the system of simple strands of $S$ the \textbf{strand depth of $F$ with respect to $S$}, denoted $d_S(F)$.
\end{itemize}
\end{definition}

We are interested in the interaction between curve depth and integral depth for balanced curve systems.

\begin{definition}[Balanced curve systems]\label{def: balanced}
A curve system $C$ on a half-plane graph $G$ is {\bf balanced} if for every edge $e$ of $G$ with endpoints $u$ and $v$, the number of curves of $C$ oriented from $u$ to $v$ along $e$ differs by at most $1$ from the number oriented from $v$ to $u$ along $e$.
\end{definition}

The alternating-sign condition of Definition~\ref{def: stranding} forces the numbers of strands directed with and against each edge to differ by at most one. Hence any system of simple strands arising from a valid stranding on a web is balanced.

\begin{lemma}\label{lem: depth bound}
Let $C$ be a balanced curve system on $G$, and let $F$ be a face of $G$. Then $|d_{C}(F)| \le d(F)$.
\end{lemma}

\begin{proof}
Let $U = F_0, F_1, \ldots, F_d = F$ be a path of length $d = d(F)$ from $U$ to $F$, crossing edges $e_1,\ldots,e_d$ where $e_i$ separates $F_{i-1}$ and $F_i$. Fix $i$ and some curve $\gamma$ of $C$ that runs along $e_i$. Since $\gamma$ is oriented, $F_{i-1}$ is either to the left or right of $\gamma$ at $e_i$. If it is to the left, either $\gamma$ is clockwise and encloses $F_{i}$ but not $F_{i-1}$ or $\gamma$ is counterclockwise and encloses $F_{i-1}$ but not $F_{i}$. In either case, $\gamma$ contributes $+1$ to $d_C(F_i)-d_C(F_{i-1})$. Similarly, if $F_{i-1}$ is to the right of $\gamma$, then $\gamma$ contributes $-1$ to $d_C(F_i)-d_C(F_{i-1})$. It follows that $d_C(F_i)-d_C(F_{i-1})$ is the difference between the number of curves of $C$ along $e_i$ for which $F_{i-1}$ is on the left and the number for which $F_{i-1}$ is on the right. Since $C$ is a balanced curve system, this implies $|d_{C}(F_i) - d_{C}(F_{i-1})| \le 1$ for each $i$. Because $d_{C}(U) = 0$, we have $$|d_C(F)|\leq \sum_{i=1}^d |d_C(F_i)-d_C(F_{i-1})|\leq d=d(F).$$
\end{proof}

\begin{example}\label{ex: depth running example}
Figure~\ref{fig: depth running example} shows integral depth in the  web from Figure~\ref{fig: running web} (left) alongside strand depth with respect to the stranding from Figure~\ref{fig: running stranding} (right), illustrating Lemma~\ref{lem: depth bound}.
\begin{figure}[h]
\centering
\begin{tikzpicture}[scale=.75, yscale=-1]
\draw[style=dashed, <->] (.5,.25)--(8.5,.25);

\begin{scope}[thick,decoration={markings,mark=at position 0.5 with {\arrow{<}}}]
\draw[postaction={decorate}] (1,.25)--(2,1);
\draw[postaction={decorate}] (2,.25)--(2,1);
\draw[postaction={decorate}] (3,.25)--(3.25,1);
\draw[postaction={decorate}] (4,.25)--(3.75,1);
\draw[postaction={decorate}] (3.25,2)--(3.25,1);
\draw[postaction={decorate}] (3.25,2)--(4.5,2.5);
\draw[postaction={decorate}] (8,.25) to[out=120, in=0] (4.5,2.5);
\draw[postaction={decorate}] (4.5,1.5)--(3.75,1);
\draw[postaction={decorate}] (4.5,1.5)--(5,1);
\draw[postaction={decorate}] (5,.25)--(5,1);
\draw[postaction={decorate}] (6,.25)--(6.5,1);
\draw[postaction={decorate}] (7,.25)--(6.5,1);
\end{scope}

\begin{scope}[decoration={markings,mark=at position 0.5 with {\arrow{<}}}]
\draw[double, postaction={decorate}] (2,1)--(3.25,2);
\draw[double, postaction={decorate}] (3.25,1)--(3.75,1);
\draw[double, postaction={decorate}] (5,1)--(6.5,1);
\draw[double, postaction={decorate}] (4.5,1.5)--(4.5,2.5);
\end{scope}

\foreach \x in {1,...,8}
  \draw[radius=.08, fill=black](\x,.25)circle;
\draw[radius=.08, fill=black](2,1)circle;
\draw[radius=.08, fill=black](3.25,1)circle;
\draw[radius=.08, fill=black](3.75,1)circle;
\draw[radius=.08, fill=black](5,1)circle;
\draw[radius=.08, fill=black](6.5,1)circle;
\draw[radius=.08, fill=black](4.5,1.5)circle;
\draw[radius=.08, fill=black](3.25,2)circle;
\draw[radius=.08, fill=black](4.5,2.5)circle;
\node[depthcolor] at (1.75,.5) {\tiny{$1$}};
\node[depthcolor] at (2.5,1) {\tiny{$1$}};
\node[depthcolor] at (3.5,.4) {\tiny{$2$}};
\node[depthcolor] at (4.5,.8) {\tiny{$2$}};
\node[depthcolor] at (3.75,1.75) {\tiny{$1$}};
\node[depthcolor] at (5.5,2) {\tiny{$1$}};
\node[depthcolor] at (5.65,.5) {\tiny{$2$}};
\node[depthcolor] at (6.5,.55) {\tiny{$2$}};
\end{tikzpicture}
\hspace{.5in}
\begin{tikzpicture}[scale=.75, yscale=-1]
\draw[style=dashed, <->] (.5,.25)--(8.5,.25);
\begin{scope}[thick,decoration={markings,mark=at position 0.5 with {\arrow{<}}}]
\draw[blue,postaction={decorate}] (1,.25)--(2,1);
\draw[blue,postaction={decorate}] (1.97,1)--(1.97,.25);
\draw[red,postaction={decorate}] (2.03,.25)--(2.03,1);
\draw[green!60!black, postaction={decorate}] (2.97,.25)--(3.22,1);
\draw[red, postaction={decorate}] (3.28,1)--(3.03,.25);
\draw[red, postaction={decorate}] (3.97,.25)--(3.72,1);
\draw[blue, postaction={decorate}] (3.78,1)--(4.03,.25);
\draw[blue, postaction={decorate}] (3.25,2)--(3.25,1);
\draw[red, postaction={decorate}] (3.25,2.04)--(4.5,2.54);
\draw[blue, postaction={decorate}] (4.5,2.47)--(3.25,1.97);
\draw[green!60!black, postaction={decorate}] (4.5,2.5) to[out=0, in=120] (8,.25);
\draw[green!60!black, postaction={decorate}] (3.75,1)--(4.5, 1.5);
\draw[red,postaction={decorate}] (4.5, 1.45)--(5,.95);
\draw[blue,postaction={decorate}] (5, 1.05)--(4.5,1.55);
\draw[green!60!black, postaction={decorate}] (5.03, .25)--(5.03,1);
\draw[red, postaction={decorate}] (4.97, 1)--(4.97,.25);
\draw[green!60!black, postaction={decorate}] (6.5, 1)--(6,.25);
\draw[blue, postaction={decorate}] (7, .25)--(6.5,1);
\end{scope}
\begin{scope}[decoration={markings,mark=at position 0.5 with {\arrow{<}}}]
\draw[red, thick, postaction={decorate}] (2,1)--(3.25,2);
\draw[blue, thick, postaction={decorate}] (3.25,1.07)--(3.75,1.07);
\draw[red, thick, postaction={decorate}] (3.75,1)--(3.25,1);
\draw[green!60!black, thick, postaction={decorate}] (3.25,.93)--(3.75,.93);
\draw[blue, thick, postaction={decorate}] (6.5,1.03)--(5,1.03);
\draw[green!60!black, thick, postaction={decorate}] (5,.97)--(6.5,.97);
\draw[blue, thick, postaction={decorate}] (4.45,1.5)--(4.45,2.5);
\draw[red, thick, postaction={decorate}] (4.5,2.5)--(4.5,1.5);
\draw[green!60!black, thick, postaction={decorate}] (4.55,1.5)--(4.55,2.5);
\end{scope}
\foreach \x in {1,...,8}
  \draw[radius=.08, fill=black](\x,.25)circle;
\draw[radius=.08, fill=black](2,1)circle;
\draw[radius=.08, fill=black](3.25,1)circle;
\draw[radius=.08, fill=black](3.75,1)circle;
\draw[radius=.08, fill=black](5,1)circle;
\draw[radius=.08, fill=black](6.5,1)circle;
\draw[radius=.08, fill=black](4.5,1.5)circle;
\draw[radius=.08, fill=black](3.25,2)circle;
\draw[radius=.08, fill=black](4.5,2.5)circle;
\node[depthcolor] at (1.75,.5) {\tiny{$1$}};
\node[depthcolor] at (2.5,1) {\tiny{$1$}};
\node[depthcolor] at (3.5,.4) {\tiny{$1$}};
\node[depthcolor] at (4.5,.8) {\tiny{$1$}};
\node[depthcolor] at (3.75,1.75) {\tiny{$0$}};
\node[depthcolor] at (5.5,2) {\tiny{$1$}};
\node[depthcolor] at (5.65,.5) {\tiny{$1$}};
\node[depthcolor] at (6.5,.55) {\tiny{$0$}};
\end{tikzpicture}
\caption{Integral depth (left) versus strand depth (right) in a web.}\label{fig: depth running example}
\end{figure}
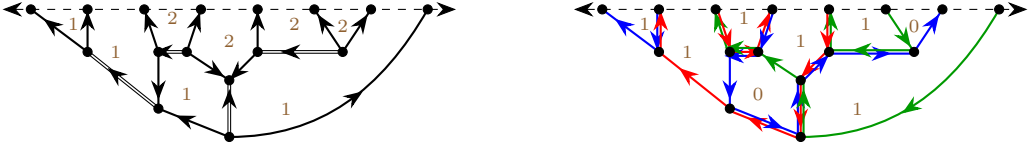
\end{example}

\section{Web vectors and web bases}\label{section: web vectors}

\subsection{The invariant space}\label{subsection: invariants}

Denote by $\uq$ the quantum universal enveloping algebra of the Lie algebra $\mathfrak{sl}_n$. Let $\mathbb{C}^n_q$ be the $n$-dimensional vector space over $\mathbb{C}(q)$ with standard basis $x_1,\ldots,x_n$. For $1\le k<n$, the $k$th fundamental representation of $\uq$ is the quantum exterior power

$$
V_k:={\bigwedge}_q^k \mathbb{C}^n_q.
$$

For a binary vector $\vec{b} = b_1 \cdots b_n \in \{0,1\}^n$ with $|\vec{b}| = k$, we define
$$
x_{\vec{b}} := x_{j_1} \wedge_q \cdots \wedge_q x_{j_k},
$$
where $\{j_1 < \cdots < j_k\} = \{i : b_i = 1\}$. For example, under this convention, $x_{1010} = x_1 \wedge_q x_3 \in V_2$. The collection $\{x_{\vec{b}} : |\vec{b}| = k\}$ forms a basis of $V_k$. We call each $x_{\vec{b}}$ a monomial. 

For $\vec{k}=(k_1,\ldots,k_m)$ such that $1\leq k_i < n$ for each $i$, write
$$
V_n(\vec{k}):=V_{k_1}\otimes\cdots\otimes V_{k_m}.
$$ The tensor products $x_{\vec{b}_1}\otimes \cdots \otimes x_{\vec{b}_m}$ with $|\vec{b}_i|=k_i$ form a basis for $V_n(\vec{k})$. We also refer to these as monomials.

A vector $w \in V_n(\vec{k})$ is \emph{$\uq$-invariant} if $\uq$ acts trivially on $w$; write $\textup{Inv}_n(\vec{k})$ for the subspace of $\uq$-invariants in $V_n(\vec{k})$. Webs serve as a complete diagrammatic model for $\textup{Inv}_n(\vec{k})$, and every construction, example, and proof in this paper proceeds at the level of webs and strandings rather than the $\uq$-action, whose explicit description can be found in \cite[Section~3.4]{RTStranding}.

The dimension of the invariant space has a clean combinatorial description in
terms of the following tableaux.

\begin{definition}[Row-strict tableaux]\label{def: RST}
Fix an integer $n\geq 2$ and a tuple $\vec{k}=(k_1,\ldots,k_m)$ with
$1\leq k_j\leq n-1$. When $n\mid\sum_j k_j$, a {\bf row-strict Young tableau}
of type $\vec{k}$ is a filling of the $n$-row rectangular Young diagram of content
$\{1^{k_1},\ldots,m^{k_m}\}$, whose entries strictly increase from left to
right along each row and weakly increase from top to bottom down each column.
Write $\mathcal{RST}_n(\vec{k})$ for the set of such tableaux. When
$n\nmid\sum_j k_j$, we set $\mathcal{RST}_n(\vec{k})=\emptyset$ by convention.
\end{definition}

\begin{theorem}[{\cite{MR1321638}}]\label{thm: invariant space dimension}
The invariant space $\textup{Inv}_n(\vec{k})$ has dimension
$|\mathcal{RST}_n(\vec{k})|$.
\end{theorem}

\subsection{Web vectors}\label{subsection: web vector}

Corresponding to each web $G\in F_n(\vec{k})$, we construct an invariant vector $f^{\uparrow}(G)$ via a state sum over the strandings of $G$.

\begin{definition}[Web Vectors]\label{def:vector}
Let $\vec{k}=(k_1,\ldots, k_m)$ and $G\in F_n(\vec{k})$ with boundary vertices $v_1,\ldots,v_m$ and incident boundary edges $e_1,\ldots,e_m$. 
\begin{itemize}
    \item Let $S\in\mathcal{S}tr(G)$. 
    \begin{itemize}
        \item For each boundary edge $e_j$, set
$$
\hat{b}_S(e_j)=\begin{cases}
b_S(e_j) & \textup{if } \sigma_{v_j}(e_j)=+1,\\
\vec{1}-b_S(e_j) & \textup{if } \sigma_{v_j}(e_j)=-1,
\end{cases}
$$
so that $|\hat{b}_S(e_j)|=k_j$. The {\bf strand monomial} of $S$ is
$$
x_S := x_{\hat{b}_S(e_1)}\otimes\cdots\otimes x_{\hat{b}_S(e_m)}\in V_n(\vec{k}).
$$
\item Let $\mathrm{cw}(S)$ be the number of closed clockwise $(i,j)$-strands of $S$ and $\mathrm{ccw}(S)$ be the number of counterclockwise $(i,j)$-strands of $S$, open or closed, where both counts range over all pairs $1\leq i<j\leq n$. Note that open clockwise
strands contribute to neither count. (This convention follows \cite[Section~5]{RTStranding}.)
We call $\mathrm{cw}(S)-\mathrm{ccw}(S)$ the {\bf strand exponent} of $S$.
\end{itemize}
\item The {\bf web vector} of $G$ is
\begin{equation*}
f^{\uparrow}(G) := \sum_{S\in\mathcal{S}tr(G)} (-q)^{\mathrm{cw}(S)-\mathrm{ccw}(S)}\,x_S.
\end{equation*}
\end{itemize}
\end{definition}

\begin{remark}
Our choice to sort wedge product factors of $x_{\vec{b}}$ increasingly differs from the decreasing convention used in \cite{RTStranding}. This scales the web vector formula from that paper by a global power of $q$ that depends solely on $\vec{k}$. To maintain expositional clarity, we adopt the ascending basis throughout this paper. To minimize ambiguity, we write $f^{\uparrow}(G)$ where $f^{\uparrow}$ is the map $f$ given in \cite{RTStranding}  scaled according to the appropriate global power.
\end{remark}
\begin{theorem}[{\cite[Theorem 50, Corollary 80]{RTStranding}}]\label{thm: web vector}
For each web $G\in F_n(\vec{k})$, we have $f^{\uparrow}(G)\in \textup{Inv}_n(\vec{k})$.
Moreover, the linear extension of $f^{\uparrow}$ to the free $\mathbb{C}(q)$-vector
space on $F_n(\vec{k})$ surjects onto $\textup{Inv}_n(\vec{k})$.
\end{theorem}

The kernel of $f^{\uparrow}$ is described in \cite[Theorem 83]{RTStranding} by a finite list of local relations on web graphs. We use exactly one of these, the edge flip relation, and only in Section~\ref{section: sl3 stranding}.

\begin{example}\label{ex: monomial}
All boundary edges in Figure~\ref{fig: running web} are directed into their boundary vertices, so $\sigma_{v_j}(e_j)=+1$ and $\hat{b}_S(e_j)=b_S(e_j)$ for every $j$. Reading the boundary binary labels of $S$ from Figure~\ref{fig: running stranding} left to right gives $1000, 0100, 0010, 0100, 0010, 0001, 1000, 0001$, so the strand monomial of $S$ is
$$
x_S = x_1\otimes x_2\otimes x_3\otimes x_2\otimes x_3\otimes x_4\otimes x_1\otimes x_4 \in V_1^{\otimes 8}.
$$
By Theorem~\ref{thm: nonzero coefficient} below, $x_S$ has nonzero coefficient in $f^{\uparrow}(G)$.
\end{example}

Our approach to analyzing web vectors rests on the following theorem, which guarantees that no 
stranding term cancels in $f^{\uparrow}(G)$.

\begin{theorem}[\cite{RTStranding}, Theorem~51]\label{thm: nonzero coefficient}
For any $G\in F_n(\vec{k})$ and any $S\in\mathcal{S}tr(G)$, the monomial 
$x_S$ has nonzero coefficient in $f^{\uparrow}(G)$.
\end{theorem}

\subsection{Web bases and leading terms}\label{subsection: web bases}

Having attached a vector to each web, we now ask when a set of webs yields a basis of
$\textup{Inv}_n(\vec{k})$.

\begin{definition}[Web basis]\label{def: web basis}
A {\bf web basis} for $\textup{Inv}_n(\vec{k})$ is a set
$\{G_1,\ldots,G_N\}\subseteq F_n(\vec{k})$ for which
$\{f^{\uparrow}(G_1),\ldots,f^{\uparrow}(G_N)\}$ is a basis of
$\textup{Inv}_n(\vec{k})$.
\end{definition}
To detect web bases, we order the monomial basis of $V_n(\vec{k})$
lexicographically.
\begin{definition}\label{def: lex order}
Fix an integer $n\geq 2$ and a tuple $\vec{k}=(k_1,\ldots,k_m)$ with
$1\leq k_j\leq n-1$.
\begin{itemize}
\item For a binary vector $\vec{b}$, let $\mathrm{seq}(\vec{b})$ be the
increasing sequence of positions of the $1$s in $\vec{b}$. The {\bf
lexicographic order} $\prec$ {\bf on the monomial basis of $V_{k_i}$} is induced
by the lexicographic order on these sequences: for instance
$x_{1100}\prec x_{1010}$, since $\mathrm{seq}(1100)=(1,2)\prec(1,3)=\mathrm{seq}(1010)$.
\item The {\bf lexicographic order} $\prec$ {\bf on the monomial basis of
$V_n(\vec{k})=V_{k_1}\otimes\cdots\otimes V_{k_m}$} is induced by the
lexicographic order on the concatenated sequences
$\mathrm{seq}(\vec{b}_1)\cdots\mathrm{seq}(\vec{b}_m)$: for instance
$x_{1001}\otimes x_{1100}\prec x_{1001}\otimes x_{1010}$.
\item The {\bf leading term} $\mathrm{lt}(w)$ of a nonzero $w\in V_n(\vec{k})$
is its $\prec$-smallest monomial with nonzero coefficient.
\item Let $G\in F_n(\vec{k})$. If $S^{\star}\in\mathcal{S}tr(G)$ such that
$\mathrm{lt}(f^{\uparrow}(G))=x_{S^{\star}}$, we call $S^{\star}$ a {\bf leading
term stranding} of $G$.
\end{itemize}
\end{definition}

Next, we show that all open strands of a leading term stranding are clockwise.

\begin{lemma}\label{lem: open strands clockwise}
Let $G\in F_n(\vec{k})$ and let $S^{\star}\in\mathcal{S}tr(G)$ be a leading term stranding of $G$. Then every open strand of $S^{\star}$ is clockwise.
\end{lemma}

\begin{proof}
Suppose $S^\star$ has a counterclockwise open $(i,j)$-strand $\gamma$, with endpoints $v_a$ and $v_b$, where $1\leq i<j\leq n$ and $a<b$, so that $\gamma$ has its tail at $v_a$ and head at $v_b$. Let $S'$ be the stranding obtained by reversing $\gamma$ in $S^\star$, which is valid by Lemma~\ref{lem: strand reversal}.

Since $\gamma$ has its tail at $v_a$ and head at $v_b$ with $a<b$, the monomial $x_{S^\star}$ contains $x_j$ (not $x_i$) in tensor factor $a$ and $x_i$ (not $x_j$) in tensor factor $b$. The strand monomial $x_{S'}$ swaps these, placing $x_i$ in factor $a$ and $x_j$ in factor $b$. Since $i<j$, we have $x_{S'}\prec x_{S^\star}$.

By Theorem~\ref{thm: nonzero coefficient}, $x_{S'}$ has nonzero coefficient in $f^{\uparrow}(G)$, contradicting the assumption that $x_{S^\star}$ is the leading term of $f^{\uparrow}(G)$. Hence every open strand of $S^\star$ is clockwise.
\end{proof}

\begin{example}
    Figure~\ref{fig: strand reversal example} shows a stranding $S'$ obtained by reversing a counterclockwise open strand in a stranding $S$. The strand monomials $x_S$ and $x_{S'}$ are identical except in the
third and fourth tensor factors where $x_S$ has $x_3\otimes x_2$  and $x_{S'}$ has $x_2\otimes x_3$. Thus $x_{S'}\prec x_S$.
\end{example}

Theorem~\ref{thm: invariant space dimension} computes the dimension of $\textup{Inv}_n(\vec{k})$ using
row-strict tableaux. To make leading terms compatible with this setup, we encode the
strand monomial with a tableau.

\begin{definition}[Tableau of a stranding]\label{def: tableau of stranding}
For $G\in F_n(\vec{k})$ and $S\in\mathcal{S}tr(G)$, define the {\bf tableau of $S$}, denoted by $T_S$, by placing $j$ in row $c$ (for $1\leq c\leq n$, $1\leq j\leq m$) exactly when $x_c$ occurs in the $j$th tensor factor of $x_S$ and arranging the entries of each row in increasing order from left to right.
\end{definition}

\begin{theorem}\label{thm: leading term row strict}
Let $G\in F_n(\vec{k})$, and let $S^{\star}$ be a leading term stranding of $G$. Then $T_{S^{\star}}\in\mathcal{RST}_n(\vec{k})$.
\end{theorem}

\begin{proof}
For $1\leq c\leq n$ and $1\leq j\leq m$, let $\rho_c(j)$ denote the number of entries of row $c$ of $T_{S^\star}$ that are at most $j$. Since $x_c$ is in the $j'$th tensor factor of $x_{S^\star}$ exactly when $\hat{b}_{S^\star}(e_{j'})_c=1$, we have
$$
\rho_c(j) = \sum_{j'=1}^j \hat{b}_{S^\star}(e_{j'})_c.
$$

Fix $1\leq c<n$. By Definition~\ref{def: ij-strands}, $v_{j'}$ is the head of a color-$c$ simple strand of $S^\star$ exactly when $\hat{b}_{S^\star}(e_{j'})_c=1$ and $\hat{b}_{S^\star}(e_{j'})_{c+1}=0$ and the tail exactly when $\hat{b}_{S^\star}(e_{j'})_c=0$ and $\hat{b}_{S^\star}(e_{j'})_{c+1}=1$. Therefore, the quantity
$$
\rho_c(j)-\rho_{c+1}(j)= \sum_{j'=1}^j (\hat{b}_{S^\star}(e_{j'})_c-\hat{b}_{S^\star}(e_{j'})_{c+1})$$
is exactly the number of heads minus the number of tails of color-$c$ simple strands meeting the boundary to the left of $v_{j+1}$. 

By Lemma \ref{lem: open strands clockwise}, every open strand of $S^\star$ is clockwise meaning the head of each strand is to the left of its tail. Therefore, $
\rho_c(j)-\rho_{c+1}(j)\geq 0$ for all $1\leq j\leq m$, and the columns of $T_{S^\star}$ weakly increase from top to bottom. Every open strand has both a head and tail, so $\rho_c(m)-\rho_{c+1}(m)= 0$ meaning all rows of $T_{S^\star}$ have the same length. 
Since $\hat{b}_{S^\star}(e_j)$ has exactly $k_j$ nonzero entries, the value $j$ appears in exactly $k_j$ rows of $T_{S^\star}$, so $T_{S^\star}$ has content $\{1^{k_1},\ldots,m^{k_m}\}$. This proves $T_{S^\star}\in\mathcal{RST}_n(\vec{k})$.
\end{proof}

\begin{remark}
While Lemma~\ref{lem: open strands clockwise} concerns all open $(i,j)$-strands, the proof of Theorem~\ref{thm: leading term row strict} only uses that open simple strands are clockwise. There exist strandings whose
open simple strands are all clockwise but which admit at least one
counterclockwise induced open $(i,j)$-strand. While the tableaux for such strandings are row-strict, by
Lemma~\ref{lem: open strands clockwise} these are not leading term
strandings. Whether every open strand (simple and induced) being clockwise is a
sufficient condition for identifying leading term strandings remains open.
\end{remark}

By Theorem~\ref{thm: invariant space dimension} and a standard triangularity argument, we get the following sufficient condition for a web basis.

\begin{corollary}\label{cor: basis criterion via tableaux}
Let $B\subseteq F_n(\vec{k})$ such that $|B|=|\mathcal{RST}_n(\vec{k})|$, and for $G\in B$ let $S^\star(G)$ be a leading term stranding of $G$. If $\{T_{S^\star(G)}:G\in B\}=\mathcal{RST}_n(\vec{k})$, then $B$ is a web basis for $\textup{Inv}_n(\vec{k})$.
\end{corollary}

\section{A web basis from colored noncrossing matchings}\label{section: matchings}

In this section, we construct a set of webs directly from a set of row-strict tableaux by way of an intermediate object called a colored noncrossing matching. The curves of the matching equip each web with a stranding. We then show the stranding coming from the matching is a leading term stranding. Corollary~\ref{cor: basis criterion via tableaux} shows the webs form a web basis.

Our construction generalizes the web--tableau correspondences of Tymoczko for
$\mathfrak{sl}_3$ and three-row standard Young tableaux \cite{TSimpleBij}, and
Russell for $\mathfrak{sl}_3$ and row-strict three-row tableaux
\cite{RussellSemistandard}, both of which replace the recursive algorithm of Khovanov
and Kuperberg \cite{KK} with a direct one. For general $n$, the noncrossing matchings implicit in the growth algorithms of
Fontaine \cite{FON} and, in its $\mathfrak{sl}_n$ restriction, Westbury
\cite{WestburyBases} satisfy conditions (a)--(d) of Step (3), so the web bases those
algorithms produce arise from Construction~\ref{def:tableau matching}.

\subsection{From a row-strict tableau to a colored matching}\label{subsection: matching construction}

The arcs constructed below become the simple strands of the web built in
Section~\ref{subsection: stranded web}; constructing the arcs first allows the stranding to
dictate the structure of the web rather than the reverse.

\begin{construction}\label{def:tableau matching}
Let $n\geq 2$ and $\vec{k}=(k_1,\dots,k_m)$ with $1\leq k_i<n$ for all $i$, and let
$T\in\mathcal{RST}_n(\vec{k})$. Construct a collection of oriented, colored arcs embedded below the boundary axis, with endpoints on the boundary axis, as follows.
\begin{enumerate}
  \item \textbf{Endpoints.}
    Place boundary vertices $v_1,\dots,v_m$ from left to right. For each $i$ and each
    color $c\in\{1,\dots,n-1\}$, if $i$ occurs in exactly one of rows $c$ and $c+1$ of
    $T$, attach to $v_i$ a color-$c$ \emph{endpoint}, of type ``in" if $i$ lies in
    row $c$ and of type ``out" if $i$ lies in row $c+1$. A vertex may carry endpoints of several colors but at most one
    endpoint of each color.

\item \textbf{Pairing.}
    For each $c$, the color-$c$ endpoints admit a unique noncrossing matching
    pairing every in-endpoint with an out-endpoint to its right: reading left to
    right, the row-strictness and column weak increase of $T$ guarantee the color-$c$
    in- and out-endpoints form a balanced sequence, and
    such a sequence has a unique noncrossing matching. Write $(v_a,v_b)$ for
    its pairs where the in-endpoint is $v_a$ and the out-endpoint is $v_b$, so $a<b$.

  \item \textbf{Arcs.}
  For each pair $(v_a,v_b)$ of color $c$, draw an arc of color $c$ below the axis,
    oriented from the out-endpoint $v_b$ to the in-endpoint $v_a$, so that:
    \begin{enumerate}
      \item arcs of the same color are disjoint;
\item at each vertex, reading the incident arcs from left to right below the axis,
  every out-arc precedes every in-arc, and the out-arc colors increase while the in-arc
  colors decrease (or the out-arc colors decrease and the in-arc colors increase);
      \item away from their endpoints, two arcs of different colors cross at most once; and
      \item away from the boundary, no point lies on more than two arcs.
    \end{enumerate}
\item \textbf{Graph and curve system.}
The arcs induce an oriented half-plane graph $M_T$ where  arc-crossings are 4-valent interior vertices, boundary vertices are $v_1,\ldots, v_m$, and each arc is broken into a sequence of edges inheriting the arc's orientation. On $M_T$, the colored arcs form a curve system which we denote by $C_T$. Since every edge of $M_T$ carries a single curve of $C_T$, the system $C_T$ is balanced. 
\end{enumerate}
We call $M_T$ (equipped with $C_T$) a \textbf{colored noncrossing matching} for $T$. 
\end{construction}

Step (3) asserts that arcs with these properties can be drawn. We verify this now.

\begin{lemma}\label{lem:existence of M_T}
The embedded arcs satisfying conditions (a)--(d) of Step (3) in Construction~\ref{def:tableau matching} exist.
\end{lemma}

\begin{proof}
Let $T\in\mathcal{RST}_n(\vec{k})$, and say that steps (1) and (2) of Construction~\ref{def:tableau matching} generate $t$ arcs with $2t$ endpoints. On a horizontal line below the boundary axis, place $2t$ equally-spaced points. Reading left to right, assign these points to $v_1,\cdots,v_m$ in turn, so that each $v_i$ receives a consecutive block of points, one for each endpoint incident to $v_i$. Join every point to its vertex by a straight segment, coloring and orienting the segments in or out so that within each $v_i$'s block the colors at $v_i$ satisfy condition (b).

Draw the ray of slope $-1$ from the open end of each in-segment and the ray of slope $+1$ from each out-segment into the lower half-plane. Where a color-$c$ in-ray first meets a color-$c$ out-ray, terminate both, yielding a color-$c$ arc oriented from its out-point to its in-point. Since the rays originate from distinct points and all have slope $\pm 1$, no more than two rays meet at any point (condition (d)). Same-color arcs are disjoint by construction (condition (a)). Since arcs are formed from two straight segments, arcs of different colors cross at most once (condition (c)).
\end{proof}

In the figures, we depict a colored noncrossing matching by drawing only the curve system $C_T$; the graph $M_T$ is left implicit.

\begin{example}
Let $n=5$ and $\vec{k}=(1,1,2,2,2,2)$, and consider the following $T\in \mathcal{RST}_5(\vec k)$.
\[
  \raisebox{-25pt}{$T=\;\;$}\ytableaushort{13,24,35,46,56}
\]
The entries $1,\dots,6$ give boundary vertices $v_1,\dots,v_6$, and the colors
$c\in\{1,2,3,4\}$ are drawn blue, red, green, and violet, and each arc is tagged with its color. Vertex endpoints are as follows
\[
  v_1\colon \text{in }1;\quad
  v_2\colon \text{in }2,\ \text{out }1;\quad
  v_3\colon \text{in }1,3,\ \text{out }2;
\]
\[
  v_4\colon \text{in }2,4,\ \text{out }1,3;\quad
  v_5\colon \text{in }3,\ \text{out }2,4;\quad
  v_6\colon \text{out }3.
\]
The associated noncrossing matchings are:
\[
  1\colon (v_1,v_2),(v_3,v_4);\quad
  2\colon (v_2,v_3),(v_4,v_5);\quad
  3\colon (v_3,v_4),(v_5,v_6);\quad
  4\colon (v_4,v_5).
\]

Figure \ref{fig:tent construction} implements the algorithm described in Lemma \ref{lem:existence of M_T}. There are two possible orderings of arc endpoints at each of vertices $v_3$, $v_4$, and $v_5$. The figure displays one possible collection of choices.

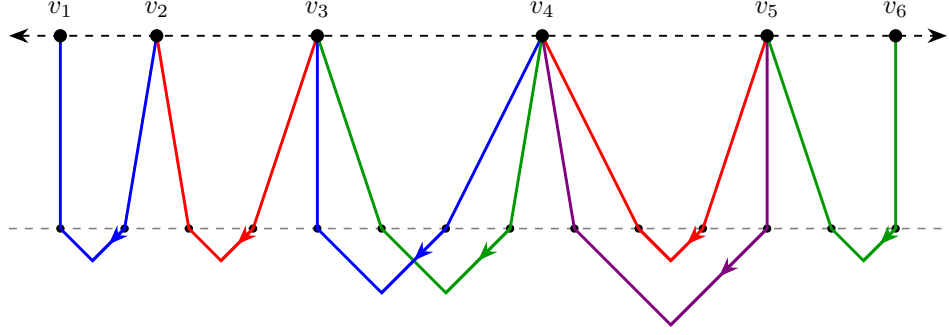
\begin{figure}
\begin{center}
\begin{tikzpicture}[
  scale=0.85,
  bv/.style={circle, fill=black, minimum size=5pt, inner sep=0pt},
  ip/.style={circle, fill=black, inner sep=1.1pt},
  ax/.style={dashed, gray, line width=0.6pt},
  font=\small
]
  \colorlet{c1}{blue}
  \colorlet{c2}{red}
  \colorlet{c3}{green!60!black}
  \colorlet{c4}{violet}

  \draw[axis,<->] (0.2,3) -- (14.8,3);
  \draw[ax]     (0.2,0) -- (14.8,0);

  \foreach \x in {1,...,14}{ \node[ip] at (\x,0) {}; }

  \draw[c1, line width=1.1pt] (2.5,3) -- (2,0)  -- (1.5,-0.5) -- (1,0)  -- (1,3);
  \draw[c2, line width=1.1pt] (5,3)   -- (4,0)  -- (3.5,-0.5) -- (3,0)  -- (2.5,3);
  \draw[c1, line width=1.1pt] (8.5,3) -- (7,0)  -- (6,-1)     -- (5,0)  -- (5,3);
  \draw[c3, line width=1.1pt] (8.5,3) -- (8,0)  -- (7,-1)     -- (6,0)  -- (5,3);
  \draw[c4, line width=1.1pt] (12,3)  -- (12,0) -- (10.5,-1.5)-- (9,0)  -- (8.5,3);
  \draw[c2, line width=1.1pt] (12,3)  -- (11,0) -- (10.5,-0.5)-- (10,0) -- (8.5,3);
  \draw[c3, line width=1.1pt] (14,3)  -- (14,0) -- (13.5,-0.5)-- (13,0) -- (12,3);

  \begin{scope}[line width=1.1pt,
      decoration={markings, mark=at position 0.5 with {\arrow{Stealth[length=2.6mm]}}}]
    \draw[c1, postaction={decorate}] (2,0)  -- (1.5,-0.5);
    \draw[c2, postaction={decorate}] (4,0)  -- (3.5,-0.5);
    \draw[c1, postaction={decorate}] (7,0)  -- (6,-1);
    \draw[c3, postaction={decorate}] (8,0)  -- (7,-1);
    \draw[c4, postaction={decorate}] (12,0) -- (10.5,-1.5);
    \draw[c2, postaction={decorate}] (11,0) -- (10.5,-0.5);
    \draw[c2] (10.5,-0.5) -- (10,0); 
    \draw[c3, postaction={decorate}] (14,0) -- (13.5,-0.5);
  \end{scope}
  \foreach \x/\y/\c in {1.5/-0.5/1, 3.5/-0.5/2, 6/-1/1, 7/-1/3, 10.5/-1.5/4, 10.5/-0.5/2, 13.5/-0.5/3}{
    \node[below=1pt, font=\scriptsize, text=c\c] at (\x,\y) {$\c$};
  }
  \foreach \x/\lab in {1/1, 2.5/2, 5/3, 8.5/4, 12/5, 14/6}{
    \node[bv] at (\x,3) {};
    \node[above=3pt] at (\x,3) {$v_{\lab}$};
  }
\end{tikzpicture}
\end{center}
\caption{Constructing $M_T$ using the algorithm from Lemma \ref{lem:existence of M_T}.}\label{fig:tent construction}
\end{figure}
\end{example}

Since every arc of $C_T$ is clockwise, curve depth $d_{C_T}(F)$ of a face $F$ of $M_T$ equals the number of arcs of $C_T$ enclosing $F$; in particular $d_{C_T}(F)\geq 0$. We now show that integral depth on $M_T$ and curve depth with respect to $C_T$ coincide.

\begin{lemma}\label{lem: adjacent smaller circle depth}
Let $M_T$ and $C_T$ be as above, and let $F$ be a face of $M_T$. Then either $F$ is the unbounded face or it shares an edge with a face of smaller curve depth.   
\end{lemma}
\begin{proof}
Since each edge of $M_T$ carries exactly one curve of $C_T$, faces separated by an edge have curve depths that differ by one. Further, since the curves of $C_T$ are oriented clockwise, the face of larger curve depth is to the right. Say $F$ is a bounded face of $M_T$ such that every neighboring face has a larger curve depth. Then the edges of $M_T$ bounding $F$ either have a  counterclockwise loop component (if $F$ is an interior face) or a counterclockwise arc component (if $F$ is a boundary face). 
    
    First consider the case that $F$ is an interior face. Let $e_1, \ldots, e_s$ be the edges of the counterclockwise loop component of $M_T$ bounding $F$ read in counterclockwise order where we define $e_{s+1}=e_1$. Each edge $e_t$ carries an arc $(v_{a_t}, v_{b_t})$ of $C_T$ with $a_t<b_t$. Adjacent edges $e_t$ and $e_{t+1}$ either carry arcs that meet at a boundary vertex in which case $a_{t+1}<b_{t+1}=a_t<b_t$ or arcs that cross in which case $a_{t+1}<a_{t}<b_{t+1}<b_t$. We therefore have $b_{s+1}<b_1$. But $b_{s+1}=b_1$, so this is not possible.

    Now say $F$ is a boundary face, and let $e_1, \ldots, e_s$ be the edges of the counterclockwise arc component of $M_T$ bounding $F$ read in counterclockwise order where $e_1$ is the initial (leftmost) edge. Using the same arc notation and inequalities as above, we have $b_s<b_1$. However, the counterclockwise order of the edges bounding $F$ means $v_{b_1}$ is to the left of $v_{a_s}$. Therefore, we have $b_1<a_s<b_s$ which is not possible.
\end{proof}

\begin{lemma}\label{lem: graph depth is circle depth}
Let $M_T$ be as above, and let $F$ be a face of $M_T$. Then $d(F)=d_{C_T}(F)$.
\end{lemma}
\begin{proof}
    By Lemma~\ref{lem: depth bound} applied to the curve system $C_T$, $d_{C_T}(F)\leq d(F)$ for every face $F$, so it suffices to prove the reverse inequality. We induct on $d_{C_T}(F)$. If $d_{C_T}(F)=0$, then Lemma~\ref{lem: adjacent smaller circle depth} forces $F$ to be the unbounded face, so $d(F)=0$. If $d_{C_T}(F)=k>0$, then by Lemma~\ref{lem: adjacent smaller circle depth} $F$ shares an edge with a face $F'$ of smaller curve depth; since faces separated by an edge differ by one in curve depth, $d_{C_T}(F')=k-1$. By induction $d(F')\leq k-1$, and crossing the shared edge yields a path from $U$ to $F$ of length at most $k$. Hence $d(F)\leq k=d_{C_T}(F)$.
\end{proof}

Finally, we make some observations about the colors and cyclic ordering of arcs of $C_T$ meeting at boundary vertices. Consider some boundary vertex of $M_T$. Denote the outgoing arc colors of $C_T$ by $h_1, \ldots, h_p$ and incoming arc colors by $r_1, \ldots, r_q$ where the colors read cyclically around the boundary vertex from left to right below the boundary axis are $h_p, \ldots, h_1, r_1,\ldots, r_q$ as shown in Figure \ref{fig: boundary arcs in tableau matching}. 

\begin{figure}[h]
\begin{tikzpicture}[scale=.5]

  \tikzset{
    arc/.style={
      draw, line width=1.1pt,
      postaction={decorate},
      decoration={markings, mark=at position 0.55 with {\arrow{Stealth[length=2.6mm]}}}
    }
  }

  \draw[axis, <->] (-5.2,0) -- (5.2,0);

  \draw[arc] ( 4.0,-1.1) .. controls ( 2.3,-1.5) .. (0,0);   
  \draw[arc] ( 1.7,-3.0) .. controls ( 0.9,-2.4) .. (0,0);   

  \draw[arc] (0,0) .. controls (-2.3,-1.5) .. (-4.0,-1.1);   
  \draw[arc] (0,0) .. controls (-0.9,-2.4) .. (-1.7,-3.0);   

  \fill (0,0) circle (2.5pt);

  \node at ( 2.55,-1.75) {$\vdots$};
  \node at (-2.55,-1.75) {$\vdots$};

  \node[anchor=west] at ( 4.10,-1.10) {$r_q$};
  \node[anchor=west] at ( 1.85,-3.10) {$r_1$};
  \node[anchor=east] at (-4.10,-1.10) {$h_p$};
  \node[anchor=east] at (-1.85,-3.10) {$h_1$};

\end{tikzpicture}
\caption{Arc colors at a boundary vertex of a colored noncrossing matching.}\label{fig: boundary arcs in tableau matching}
\end{figure}
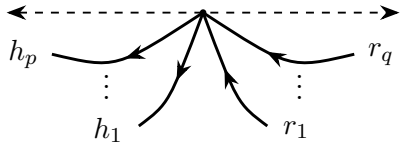

\begin{lemma}\label{lem:boundary arc alternating}
Using the notation introduced above, we have $|p-q|\leq 1$. Moreover, the arc colors satisfy one of the following: 
\begin{enumerate}[leftmargin=6cm] 
    \item[(A)] $r_1>h_1>r_2>h_2\cdots$ \;\; with $p\leq q$,
    \item[(B)] $r_1<h_1<r_2<h_2\cdots$ \;\; with  $p\leq q$,
    \item[(C)] $h_1>r_1>h_2>r_2\cdots$ \;\; with $q\leq p$,
    \item[(D)] $h_1<r_1<h_2<r_2\cdots$ \;\; with $q\leq p$.
\end{enumerate}   
\end{lemma}

\begin{proof}
Let $v_j$ be a boundary vertex of $M_T$. By Step (1) of Construction~\ref{def:tableau matching}, $v_j$ carries at most one
endpoint of each color, so there are no color repeats. The colors also alternate: an in-color
of $c$ means $j$ lies in row $c$ but not row $c+1$, and an out-color means the
reverse, so colors must necessarily switch between in and out. Hence $|p-q|\leq 1$.
The arrangement of arcs at boundary vertices described in Step (3b) then forces the
colors to form a sequence of one of the four listed types.
\end{proof}

\subsection{Forming a stranded web from a colored noncrossing matching}\label{subsection: stranded web}

From a colored noncrossing matching, we now build a web graph equipped with a stranding. In the figures illustrating this construction, color is schematic: distinct colors mark distinct arc colors $c \neq c'$ at a generic vertex or crossing, and do not refer to the fixed color assignments used in our worked examples.

\begin{construction}[Stranded web $G_T$]\label{def: G_T and S_T}
Let $T\in\mathcal{RST}_n(\vec{k})$. From a colored noncrossing matching $M_T$ for $T$ as in Construction \ref{def:tableau matching}, we build an oriented, weighted graph $G_T$ equipped with curve system $S_T$ as follows.

\begin{enumerate}
  \item \textbf{Crossings.}
    Replace a neighborhood of each crossing as shown in Figure~\ref{fig: resolution}.

    \begin{figure}[h]
                \begin{center}
                    \begin{tikzpicture}[scale=.65, line cap=round]

\begin{scope}[shift={(0,0)}]
  \crosscurve{red}{(1.5,-1.9)}{(-1.5,1.9)}{8}{0.66}   
  \crosscurve{blue}{(1.5,1.9)}{(-1.5,-1.9)}{8}{0.50}  
\end{scope}

\draw[line width=1.1pt,-{Stealth[length=3mm]},decorate,
      decoration={snake,amplitude=1mm,segment length=4mm,post length=2.5mm}]
   (3.1,0) -- (5.6,0);

\begin{scope}[shift={(8.6,0)}]
  \coordinate (U) at (0,0.95);  \coordinate (L) at (0,-0.95);
  \coordinate (UL) at (-1.4,2.0); \coordinate (UR) at (1.4,2.0);
  \coordinate (LL) at (-1.4,-2.0);\coordinate (LR) at (1.4,-2.0);
  \draw[black,line width=0.7pt] (UL)--(U);
  \draw[black,line width=0.7pt] (U)--(UR);
  \draw[black,line width=0.7pt] (L)--(LL);
  \draw[black,line width=0.7pt] (LR)--(L);
  \draw[black,line width=0.7pt] (L)--(U);
  \strandA{red}   
  \strandB{blue}  
  \fill (U) circle (2.5pt); \fill (L) circle (2.5pt);
  \node[depthcolor] at (0,-1.85) {$d-1$};
\node[depthcolor] at (-0.95,0) {$d$}; \node[depthcolor] at (0.95,0) {$d$};
\node[depthcolor] at (0,1.75) {$d+1$};
\end{scope}

\end{tikzpicture}
\caption{Resolving crossings in $M_T$ to create web $G_T$ stranded by $S_T$ with relative strand depths $d-1, d, d+1$.} \label{fig: resolution}
     \end{center}
\end{figure}
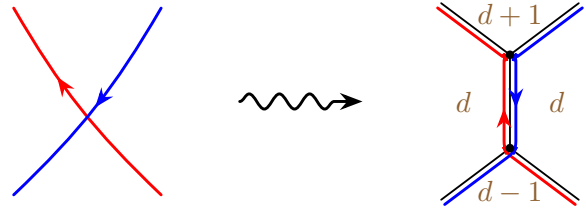  

  \item \textbf{Boundary vertices.}
    Replace a neighborhood of each boundary vertex with more than one incident edge by $p+q-1$ consecutive vertical edges (a trunk) together with edges on the left and right of the trunk (branches) corresponding to the outgoing and incoming arcs of $C_T$ meeting the trunk in the same cyclic order as in $M_T$ and ordered so that the branch colors form a monotone sequence from bottom to top, as shown in Figure~\ref{fig: trunks and branches}.

\begin{figure}[htbp]
    \centering
    
    \tikzset{
        strand/.style={line width=1.1pt, rounded corners=3mm, decoration={markings, mark=at position 0.75 with {\arrow{Stealth[length=2.6mm]}}}, postaction={decorate}}
    }

    \begin{subfigure}[b]{0.45\textwidth}
        \centering
        \resizebox{0.6\linewidth}{!}{
            \begin{tikzpicture}[line cap=round, line join=round]
                \node at (0, 1.7) {$\vdots$};
                \draw[line width=0.7pt] (0, 1.4) -- (0, 1.2);
                
                \draw[line width=0.7pt] (0, 0)--(0, 1.2);                 
                \draw[line width=0.7pt]  (1.2, -0.2) --(0,0);              
                \draw[line width=0.7pt] (0, -1)--(0,0);                  
          \draw[line width=0.7pt] (0, -1) -- (-1.2, -1.8);           
                \draw[line width=0.7pt] (1.2, -1.8) --(0, -1) ;             
                \fill (0,0) circle (2.5pt);
                \fill (0,-1) circle (2.5pt);

                \draw[violet, strand] (1.2, -0.05)  -- (0.15, 0.05) -- (0.15, 1.2);
                \draw[red, strand] (1.2, -1.65) -- (0.08, -0.95) -- (0.08, 1.2);
                \draw[blue, strand] (-0.08, 1.2) -- (-0.08, -0.95) -- (-1.2, -1.65) ;
                
               \node[depthcolor] at (0,-1.5) {$d-1$};
\node[depthcolor] at (-0.6,-0.25) {$d$}; \node[depthcolor] at (0.58,-0.62) {$d$};
\node[depthcolor] at (0.72,0.72) {$d+1$}; 
                \path (-1.8, -2.2) rectangle (1.8, 2.2);
            \end{tikzpicture}
        }
        \caption{$r_1$ extremal}
    \end{subfigure}
    \begin{subfigure}[b]{0.45\textwidth}
        \centering
        \resizebox{0.6\linewidth}{!}{
            \begin{tikzpicture}[line cap=round, line join=round]
                \node at (0, 1.7) {$\vdots$};
                \draw[line width=0.7pt] (0, 1.4) -- (0, 1.2);
                
                \draw[line width=0.7pt] (0, 1.2) -- (0, 0);                 
                \draw[line width=0.7pt] (0, 0) -- (-1.2, -0.2);             
                \draw[line width=0.7pt] (0, 0) -- (0, -1);                  
                \draw[line width=0.7pt] (0, -1) -- (-1.2, -1.8);            
                \draw[line width=0.7pt] (1.2, -1.8) --(0, -1) ;             
                \fill (0,0) circle (2.5pt);
                \fill (0,-1) circle (2.5pt);

                \draw[green!60!black, strand] (-0.15, 1.2) -- (-0.15, 0.05) -- (-1.2, -0.05);
            \draw[red, strand] (1.2, -1.65) -- (0.08, -0.95) -- (0.08, 1.2);
                \draw[blue, strand] (-0.08, 1.2) -- (-0.08, -0.95) -- (-1.2, -1.65) ;

                \node[depthcolor] at (0,-1.5) {$d-1$};
\node[depthcolor] at (0.6,-0.25) {$d$}; \node[depthcolor] at (-0.58,-0.62) {$d$};
\node[depthcolor] at (-0.75,0.72) {$d+1$};
                
                \path (-1.8, -2.2) rectangle (1.8, 2.2);
            \end{tikzpicture}
        }
        \caption{$h_1$ extremal}
    \end{subfigure}         
           \caption{The trunk and branches of $G_T$ stranded by $S_T$ replacing the neighborhood of a boundary vertex of $M_T$ as shown in Figure \ref{fig: boundary arcs in tableau matching} together with relative strand depths $d-1, d, d+1$}\label{fig: trunks and branches}
          \end{figure}
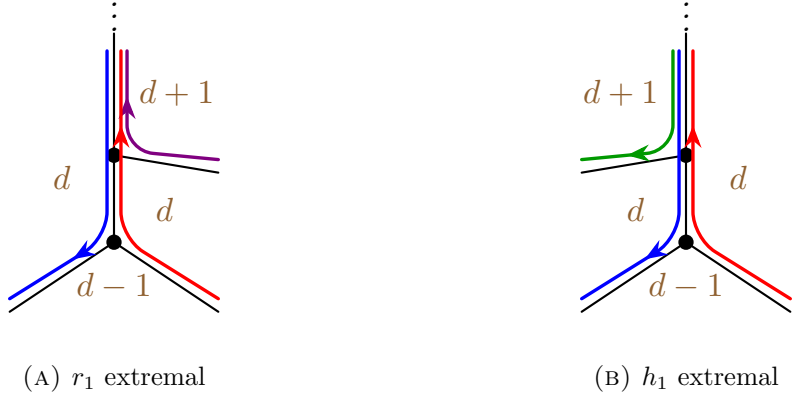

  \item \textbf{Arc interiors.}
    Treat every remaining portion of an edge, including a boundary vertex with a single incident edge, as in Figure~\ref{fig: easy web conversion}.

    \begin{figure}[h]
\begin{tikzpicture}[scale=.6, line cap=round, line width=1.1pt, >=Stealth]
 
  \tikzset{
    midarrow/.style={
      postaction={decorate},
      decoration={markings, mark=at position #1 with {\arrow{Stealth[length=2.6mm]}}}
    }
  }
 
  \draw[blue, line width=1.1pt, midarrow=0.5]
    (1.8,0.35) .. controls (0.5,-0.25) and (-0.5,-0.25) .. (-1.8,0.3);
 
  \draw[black, line width=1.1pt, -{Stealth[length=3mm]}, decorate,
        decoration={snake, amplitude=1mm, segment length=4mm, post length=2.5mm}]
    (3.1,0.05) -- (5.2,0.12);
 
  \draw[black, line width=1.1pt]
    (10.0,0.55) .. controls (8.7,-0.05) and (7.7,-0.05) .. (6.4,0.5);
  \draw[blue, line width=1.1pt, midarrow=0.38]
    (10.0,0.35) .. controls (8.7,-0.25) and (7.7,-0.25) .. (6.4,0.3);

  \node[depthcolor] at (8.2,0.95) {$d$};
  \node[depthcolor] at (8.2,-0.7) {$d-1$};
 
\end{tikzpicture}
    \caption{Converting arc interiors to stranded web edge interiors, with relative strand depths $d-1$ and $d$.}\label{fig: easy web conversion}
\end{figure}
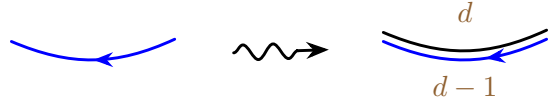

  \item \textbf{Orientation and weights.}
    Orient each edge of $G_T$ in the direction of its largest arc color in $S_T$, and weight it by the resulting alternating sum.
\end{enumerate}
\end{construction}

\begin{example}
In Figure \ref{fig:ex of matching stranding}, we see a smoothed version of the colored noncrossing matching from Figure \ref{fig:tent construction} and the stranded web built according to Construction~\ref{def: G_T and S_T}.
\begin{figure}[h] 
 \tikzset{
  strand/.style={line width=1.1pt, rounded corners=2.2mm,
      decoration={markings, mark=at position #1 with {\arrow{Stealth[length=2.6mm]}}},
      postaction={decorate}},
  arc/.style={line width=1.1pt,
      decoration={markings, mark=at position #1 with {\arrow{Stealth[length=2.6mm]}}},
      postaction={decorate}},
  graph edge/.style={black, line width=0.7pt, rounded corners=0.6mm,
      decoration={markings, mark=at position #1 with {\arrow{Stealth[length=2.6mm]}}},
      postaction={decorate}},
  bdry/.style={black, dashed, very thick},
  lab/.style={font=\small},
}

\begin{center}
 \begin{tikzpicture}[line cap=round,line join=round]
 
\begin{scope}[yshift=0cm]
  \foreach \i/\x in {1/0,2/2,3/4,4/6,5/8,6/10}{ \coordinate (m\i) at (\x,0); }
  \draw[axis, <->] (-1.2,0) -- (11.2,0);
  \draw[blue, arc=0.55] (m2) .. controls (1.6,-1.4) and (0.4,-1.4) .. (m1);
  \draw[red, arc=0.55]  (m3) .. controls (3.6,-1.4) and (2.4,-1.4) .. (m2);
  \draw[green!60!black, arc=0.70] (m4) .. controls (5.5,-1.9) and (4.7,-0.7) .. (m3);
  \draw[blue, arc=0.35] (m4) .. controls (5.3,-0.7) and (4.5,-1.9) .. (m3);
  \draw[red, arc=0.55]  (m5) .. controls (7.6,-1.4) and (6.4,-1.4) .. (m4);
  \draw[violet, arc=0.55] (m5) .. controls (8.2,-2.6) and (6.5,-2.6) .. (m4);
  \draw[green!60!black, arc=0.55] (m6) .. controls (9.6,-1.5) and (8.4,-1.5) .. (m5);
  \foreach \i in {1,...,6}{ \fill (m\i) circle (2.5pt); }
\end{scope}
 \end{tikzpicture}
 
 \begin{tikzpicture}[line cap=round,line join=round]
\begin{scope}[yshift=-5.6cm,scale=1.15]
  \coordinate (v1) at (0,0);     \coordinate (v2) at (1.91,0);
  \coordinate (v3) at (3.33,0);  \coordinate (v4) at (5.85,0);
  \coordinate (v5) at (8.04,0);  \coordinate (v6) at (11.81,0);
  \coordinate (A) at (1.96,-1.23);  \coordinate (B) at (3.41,-0.80);
  \coordinate (C) at (3.31,-1.58);  \coordinate (D) at (4.74,-1.26);
  \coordinate (F) at (4.72,-2.13);  \coordinate (Et) at (5.84,-1.08);
  \coordinate (Eb) at (5.95,-1.52); \coordinate (G) at (5.84,-2.14);
  \coordinate (H) at (8.08,-1.60);  \coordinate (I) at (8.27,-2.53);
 
  \draw[axis, <->] (-1.1,0.04) -- (12.9,0.04);
 
  \draw[blue, line width=1.1pt, line join=round, rounded corners=2.0pt,
        decoration={markings, mark=at position 0.657 with {\arrow{Stealth[length=2.6mm]}}}, postaction={decorate}]
     (1.910,0.000) -- (1.852,-0.163) -- (1.889,-1.073) -- (1.790,-1.200) -- (0.101,-0.140) -- (0.000,0.000);
  \draw[red, line width=1.1pt, line join=round, rounded corners=2.0pt,
        decoration={markings, mark=at position 0.536 with {\arrow{Stealth[length=2.6mm]}}}, postaction={decorate}]
     (3.330,0.000) -- (3.281,-0.166) -- (3.329,-0.647) -- (3.325,-0.950) -- (3.266,-1.413) -- (3.139,-1.603) -- (2.099,-1.333) -- (2.018,-1.067) -- (1.981,-0.157) -- (1.910,0.000);
  \draw[blue, line width=1.1pt, line join=round, rounded corners=2.0pt,
        decoration={markings, mark=at position 0.612 with {\arrow{Stealth[length=2.6mm]}}}, postaction={decorate}]
     (5.850,0.000) -- (5.914,-0.161) -- (5.906,-0.921) -- (5.693,-1.170) -- (4.908,-1.298) -- (4.671,-1.418) -- (4.659,-1.969) -- (4.547,-2.132) -- (3.435,-1.699) -- (3.395,-1.430) -- (3.454,-0.967) -- (3.459,-0.634) -- (3.411,-0.153) -- (3.330,0.000);
  \draw[green!60!black, line width=1.1pt, line join=round, rounded corners=2.0pt,
        decoration={markings, mark=at position 0.738 with {\arrow{Stealth[length=2.6mm]}}}, postaction={decorate}]
     (5.850,0.000) -- (5.654,-0.158) -- (5.646,-0.918) -- (5.801,-1.193) -- (5.863,-1.439) -- (5.862,-1.645) -- (5.800,-1.992) -- (5.679,-2.204) -- (4.879,-2.196) -- (4.789,-1.972) -- (4.801,-1.421) -- (4.568,-1.269) -- (3.540,-0.914) -- (3.200,-0.660) -- (3.152,-0.179) -- (3.330,0.000);
  \draw[red, line width=1.1pt, line join=round, rounded corners=2.0pt,
        decoration={markings, mark=at position 0.494 with {\arrow{Stealth[length=2.6mm]}}}, postaction={decorate}]
     (8.040,0.000) -- (7.979,-0.162) -- (8.011,-1.442) -- (7.918,-1.659) -- (6.107,-1.591) -- (5.989,-1.407) -- (5.927,-1.161) -- (6.036,-0.922) -- (6.044,-0.162) -- (5.850,0.000);
  \draw[violet, line width=1.1pt, line join=round, rounded corners=2.0pt,
        decoration={markings, mark=at position 0.524 with {\arrow{Stealth[length=2.6mm]}}}, postaction={decorate}]
     (8.040,0.000) -- (8.239,-0.155) -- (8.271,-1.435) -- (8.048,-1.770) -- (8.174,-2.386) -- (8.102,-2.569) -- (5.988,-2.230) -- (5.928,-2.015) -- (5.990,-1.668) -- (5.737,-1.470) -- (5.675,-1.224) -- (5.776,-0.919) -- (5.784,-0.159) -- (5.850,0.000);
  \draw[green!60!black, line width=1.1pt, line join=round, rounded corners=2.0pt,
        decoration={markings, mark=at position 0.323 with {\arrow{Stealth[length=2.6mm]}}}, postaction={decorate}]
     (11.810,0.000) -- (11.718,-0.146) -- (8.438,-2.490) -- (8.302,-2.360) -- (8.176,-1.744) -- (8.141,-1.438) -- (8.109,-0.158) -- (8.040,0.000);

  \draw[graph edge=0.55] (A) to[out=150,in=-82] (v1);  
  \draw[graph edge=0.5] (A) -- (v2);                    
  \draw[graph edge=0.5] (C) -- (A);                     
  \draw[graph edge=0.5] (B) -- (v3);                    
  \draw[graph edge=0.5] (B) -- (C);                     
  \draw[graph edge=0.5] (D) -- (B);                     
  \draw[graph edge=0.5] (F) -- (C);                     
  \draw[graph edge=0.5] (F) -- (D);                     
  \draw[graph edge=0.5] (Et) -- (D);                    
  \draw[graph edge=0.5] (G) -- (F);                     
  \draw[graph edge=0.5] (Et) -- (v4);                   
  \draw[graph edge=0.6] (Eb) -- (Et);                   
  \draw[graph edge=0.5] (G) -- (Eb);                    
  \draw[graph edge=0.5] (H) -- (Eb);                    
  \draw[graph edge=0.5] (I) -- (G);                     
  \draw[graph edge=0.5] (v5) -- (H);                    
  \draw[graph edge=0.5] (H) -- (I);                     
  \draw[graph edge=0.5] (v6) to[out=212,in=-18] (I);    
 
  \foreach \p in {v1,v2,v3,v4,v5,v6,A,B,C,D,Et,Eb,F,G,H,I}{ \fill (\p) circle (2.5pt); }
 
  \node[lab] at (0.95,-0.55) {$1$};   \node[lab] at (2.55,-1.75) {$2$};
  \node[lab] at (2.20,-0.34) {$1$};   \node[lab] at (3.2,-0.25) {$2$};
  \node[lab] at (4,-0.75) {$3$};   \node[lab] at (5.1,-0.95) {$1$};
  \node[lab] at (6.12,-0.30) {$2$};   \node[lab] at (3.25,-1.25) {$1$};
  \node[lab] at (4.45,-1.5) {$2$};   \node[lab] at (5.7,-1.8) {$1$};
  \node[lab] at (6.15,-1.2) {\small{$3$}};   
  \node[lab] at (7.00,-1.3) {$2$};
  \node[lab] at (3.8,-2.) {$1$};   \node[lab] at (5.3,-2.35) {$3$};
  \node[lab] at (7.30,-2.65) {$4$};   \node[lab] at (7.85,-2.02) {$1$};
  \node[lab] at (7.75,-0.62) {$3$};   \node[lab] at (10,-1.65) {$3$};
\end{scope}
\end{tikzpicture}
    \end{center}
    \caption{A colored noncrossing matching and its corresponding stranded web.}\label{fig:ex of matching stranding}
\end{figure}
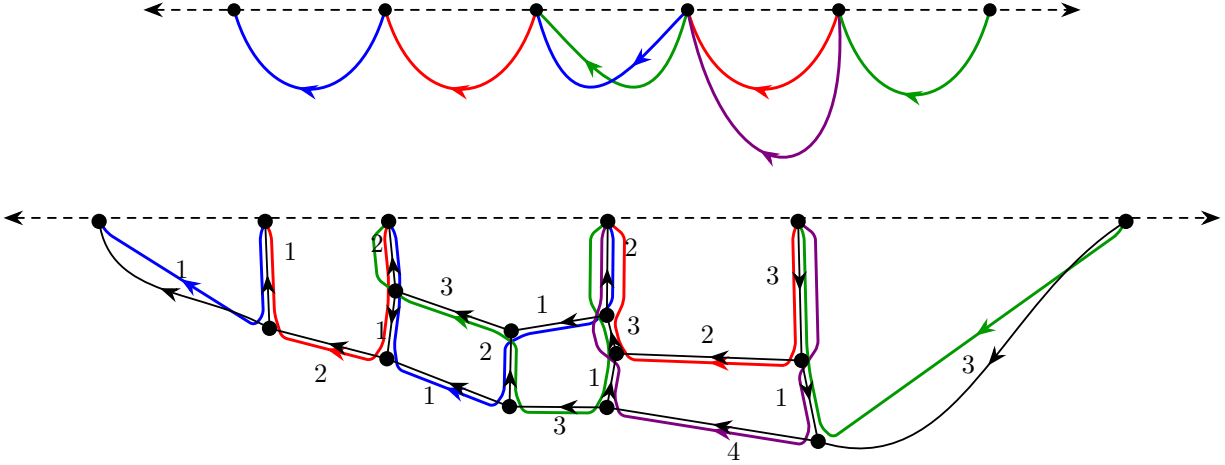

\end{example}

\begin{lemma}\label{lem: G_T is a stranded web}
Let $T\in\mathcal{RST}_n(\vec k)$. Then $G_T$ is an $\mathfrak{sl}_n$ web, and $S_T\in\mathcal{S}tr(G_T)$.
\end{lemma}

\begin{proof}

By construction, on each edge of $G_T$ the curves of $S_T$ alternate direction when ordered by color, so by Lemma~\ref{lem: alternating sum weight}, every edge of $G_T$ has a weight in $\{1,\ldots,n-1\}$ and is validly stranded by $S_T$. It remains to verify the interior vertex condition of Definition~\ref{def: webs} at each interior vertex of $G_T$.

Let $v$ be an interior vertex with incident edges $e_1, e_2, e_3$ of weights $\ell_1, \ell_2, \ell_3$. Every curve of $S_T$ incident to $v$ passes through it, entering along one incident edge and leaving along another. For every color $c$, we therefore have
$$\sum_{i=1}^3 \sigma_v(e_i)\,\alpha_c^\vee(S_T(e_i))=0.$$
Multiplying by $c$ and summing over colors gives $$\sum_{c=1}^{n-1} \sum_{i=1}^3 \sigma_v(e_i)\,\alpha_c^\vee(S_T(e_i))\,c \;=\; \sum_{i=1}^3 \sigma_v(e_i)\!\left(\sum_{c=1}^{n-1} \alpha_c^\vee(S_T(e_i))\,c\right) = 0.$$
By construction, edge $e_i$ has weight $\ell_i=\sum_{c=1}^{n-1}\alpha_c^\vee(S_T(e_i))\,c$, so  $\sum_{i=1}^3\sigma_v(e_i)\,\ell_i= 0$, and the interior vertex condition holds.
\end{proof}

Note that the proof gives $\sum_{i=1}^3 \sigma_v(e_i)\ell_i=0$ exactly, not merely
modulo $n$, at every interior vertex of $G_T$. Lemma~\ref{lem: flow condition on binary labels} therefore applies at every interior
vertex of $G_T$.

Next, we observe the stranding $S_T$ has corresponding tableau $T$.

\begin{theorem}\label{thm: monomial of S_T}
Let $T\in \mathcal{RST}_n(\vec{k})$. Then $T_{S_T}=T$.
\end{theorem}

\begin{proof}
Consider some boundary vertex $v_j$ with incident edge $e_j$. We examine the case
that $r_1<h_1<r_2<\cdots< r_q$ at $v_j$. In this case, the largest endpoint color at
$v_j$ is the in-color $r_q$, so $j$ does not lie in row $n$ of $T$, and $c\leq r_1$
or $h_i<c\leq r_{i+1}$ for $1\leq i<q$ exactly when $T$ has entry $j$ in row $c$.
By construction, $S_T$ has incoming simple strand colors $r_1, \ldots, r_q$ and
outgoing colors $h_1,\ldots, h_{q-1}$ at $v_j$. Thus the $j$th boundary binary
label coming from $S_T$ is $\hat{b}_{S_T}(e_j)=\lambda_{r_1}+\sum_{i=1}^{q-1}
(\lambda_{r_{i+1}}-\lambda_{h_i})$. Since $h_{i}<r_{i+1}$, the vector
$\lambda_{r_{i+1}}-\lambda_{h_i}$ has 1 in positions $h_{i}+1$ through $r_{i+1}$ and
0 elsewhere. Putting this together, $\hat{b}_{S_T}(e_j)$ has 1 in position $c$ if and
only if $c\leq r_1$ or $h_i<c\leq r_{i+1}$ for $1\leq i<q$. Since $x_c$ is a wedge
factor of the $j$th tensor factor of $x_{S_T}$ exactly when $\hat{b}_{S_T}(e_j)$
has 1 in position $c$, the result follows. The other cases described by Lemma
\ref{lem:boundary arc alternating} can be argued analogously; when the largest
endpoint color at $v_j$ is an out-color, $j$ lies in row $n$ and $e_j$ is oriented
out of the boundary, so the resulting complementation is absorbed by
$\hat{b}_{S_T}$.
\end{proof}

\begin{corollary}\label{cor: G_T in F_n(k)}
Let $T\in\mathcal{RST}_n(\vec k)$. Then $G_T\in F_n(\vec k)$.
\end{corollary}

\begin{proof}
Since $T$ has content $\{1^{k_1},\ldots,m^{k_m}\}$ and the rows of $T$ strictly increase, the entry $j$ occupies $k_j$ distinct rows of $T$. By Theorem~\ref{thm: monomial of S_T}, the $j$th tensor factor of $x_{S_T}$ is therefore a wedge of exactly $k_j$ basis vectors, so $|\hat{b}_{S_T}(e_j)|=k_j$. By Definition~\ref{def:vector}, this says precisely that $v_j$ carries boundary weight $k_j$. Hence $G_T$ has boundary weight vector $\vec k$.
\end{proof}

We now record some observations about depth in the stranded web $G_T$.

\begin{lemma}\label{lem: matching web depth}
    Let $G_T$ and $S_T$ be as above. Then $d(F)=d_{S_T}(F)$ for every face $F$ of $G_T$.
\end{lemma}

\begin{proof}
Let $F$ be a face of $G_T$ and $d(F)$ be the integral depth of $F$ in $G_T$. The graph $M_T$ can be obtained from $G_T$ via a sequence of edge contractions during which no self-loops are created or destroyed. This induces a bijection on the faces of $G_T$ and $M_T$ under which we can view $F$ as a face of $M_T$. The curves of $C_T$ enclosing $F$ in $M_T$ are in one-to-one correspondence with the simple strands of $S_T$ enclosing $F$ in $G_T$; hence $d_{C_T}(F)=d_{S_T}(F)$. By Lemma~\ref{lem: graph depth is circle depth}, this common value is the integral depth of $F$ in $M_T$. Since every path from $U$ to $F$ in $M_T$ is also a path in $G_T$, $d(F)$ is at most the integral depth of $F$ in $M_T$, so $d(F)\leq d_{C_T}(F)=d_{S_T}(F)$. The reverse inequality $d_{S_T}(F)\leq d(F)$ is Lemma~\ref{lem: depth bound} applied to the balanced curve system $S_T$. Hence $d(F)=d_{S_T}(F)$.
\end{proof}

By the preceding result, we can freely conflate integral depth and strand depth $d_{S_T}$ in $G_T$. We can use this to track the evolution of depth across edges in $G_T$.

\begin{lemma}\label{lem:nearby lemma}
Let $e$ be an interior edge of $G_T$ separating faces of depth $d$. The set of depths of faces incident to the endpoints of $e$ is $\{d-1,d, d+1\}$. 
\end{lemma}
\begin{proof}
    Such an edge either comes from a crossing of $M_T$ or is a trunk edge at some vertex. Since depth is the number of simple strands enclosing a face, and all strands are clockwise, we can verify this result by considering the local modifications used to construct $G_T$ from $M_T$. These are shown in Figures~\ref{fig: resolution} and~\ref{fig: trunks and branches}.
\end{proof}

Recall that $L_{(1,n)}(S_T)$ is the curve system on $G_T$ which is the union of the $(1,n)$-strands of $S_T$. 
 The following lemma lists properties of $L_{(1,n)}(S_T)$.

\begin{lemma}\label{lem: (1,n) strands}
Let $G_T$ and $S_T$ be as above. 
\begin{enumerate}
    \item $L_{(1,n)}(S_T)$ is the union of edges of $G_T$ separating faces of different depth.
    \item Every $(1,n)$-strand of $S_T$ is clockwise.
    \item Every $(1,n)$-strand of $S_T$ is open.
\end{enumerate}
\end{lemma}

\begin{proof}
Recall that an edge $e$ of $G_T$ carries a $(1,n)$-strand of $S_T$ exactly when the first and last entries of $b_{S_T}(e)$ differ. Every edge of $G_T$ is oriented in the direction of the largest colored simple strand of $S_T$, so $b_{S_T}(e)=\sum_{c=1}^{n-1}\alpha_c^\vee(S_T(e))\lambda_c$. For all $c$, the vector $\lambda_c$ has 1 in the first position and 0 in the last. Thus, $e$ carries a $(1,n)$-strand of $S_T$ exactly when the number of simple strands of $S_T$ on $e$ is odd.
\begin{enumerate}
\item  Lemma \ref{lem: matching web depth} shows that integral depth of a face in $G_T$ is the number of simple strands of $S_T$ enclosing that face. This means faces on either side of an edge have different depth precisely when the number of simple strands of $S_T$ along that edge is odd.  
\item By construction, the binary labels prescribed by $S_T$ have 0 in the $n$th entry, so every edge of $G_T$ carrying a $(1,n)$-strand of $S_T$ is directed with the $(1,n)$-strand. Therefore, along each $(1,n)$-strand, faces on the left have smaller depth than faces on the right. Each $(1,n)$-strand of $S_T$ splits the half-plane into two regions one of which contains the unbounded face. The region containing the unbounded face must be on the side of this component of smaller depth. Therefore every $(1,n)$-strand is clockwise. 

\item Suppose $S_T$ has a closed $(1,n)$-strand, and
choose an innermost one. By Lemma~\ref{lem:nearby lemma}, there are no edges of
$G_T$ inside it, so it encloses a single face. Suppose the strand contains a trunk
edge $e$. Since $e$ is not a boundary edge, there is an adjacent trunk edge carrying one more simple strand than $e$. But this would mean there is an edge within the closed $(1,n)$-strand on which $e$ lies which is not possible. Hence the strand contains no trunk edges, every vertex along
it comes from a crossing of $C_T$, and every edge of it carries a single simple
strand. As in the closed loop case of the proof of Lemma~\ref{lem: adjacent smaller circle depth}, this leads to a contradictory system of inequalities and is therefore not possible.
\end{enumerate}
 \end{proof}

By Lemma~\ref{lem: (1,n) strands}, the $(1,n)$-strands of $S_T$ are clockwise arcs, and by Lemma~\ref{lem:nearby lemma}, the edges of $G_T$ not in $L_{(1,n)}(S_T)$ (those separating faces of the same depth) have endpoints on nested $(1,n)$-strands. This gives the following structural corollary.

\begin{corollary} \label{cor: Arcs and vertical edges}
The web $G_T$ can be drawn so that the edges of $L_{(1,n)}(S_T)$ form a noncrossing
matching drawn as a collection of half circles, and every remaining edge of $G_T$ is a vertical
segment running either to the boundary axis or between two nested components of
$L_{(1,n)}(S_T)$.
\end{corollary}

\begin{example}
In Figure~\ref{fig: running example half circles}, we redraw the example web from Figure~\ref{fig:ex of matching stranding} in the form of Corollary~\ref{cor: Arcs and vertical edges}. The $(1,n)$-strands of $L_{(1,n)}(S_T)$, drawn in orange, form nested half-circles; the remaining vertical edges of $G_T$ run either to the boundary or between nested components of $L_{(1,n)}(S_T)$. Depths are shown in brown. Edge weights are omitted but can be inferred from Figure~\ref{fig:ex of matching stranding}.
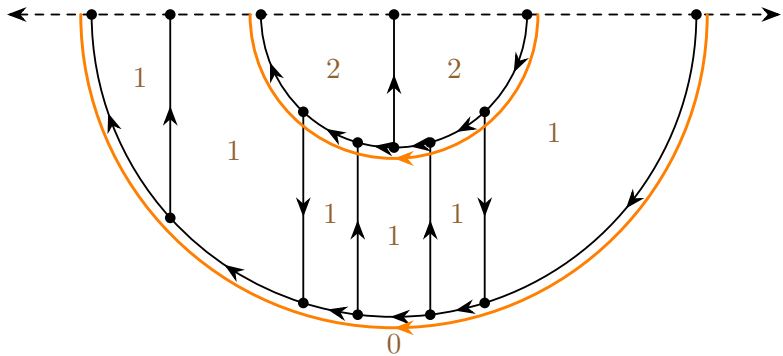
\begin{figure}[h]
\centering
\begin{tikzpicture}[scale=0.8, line cap=round, line join=round, >={Stealth[length=2.6mm]},
  gedge/.style={black, line width=0.7pt, postaction={decorate}, decoration={markings, mark=at position 0.55 with {\arrow{Stealth[length=2.6mm]}}}},
  nstr/.style={line width=1.1pt, color=orange, postaction={decorate}, decoration={markings, mark=at position 0.5 with {\arrow{Stealth[length=2.6mm]}}}}]
\draw[axis,<->] (-6.4,0)--(6.4,0);
\draw[gedge] (222.269:5.0) arc[start angle=222.269, end angle=180.000, radius=5.0];
\draw[gedge] (252.542:5.0) arc[start angle=252.542, end angle=222.269, radius=5.0];
\draw[gedge] (263.108:5.0) arc[start angle=263.108, end angle=252.542, radius=5.0];
\draw[gedge] (276.892:5.0) arc[start angle=276.892, end angle=263.108, radius=5.0];
\draw[gedge] (287.458:5.0) arc[start angle=287.458, end angle=276.892, radius=5.0];
\draw[gedge] (360.000:5.0) arc[start angle=360.000, end angle=287.458, radius=5.0];
\draw[gedge] (227.014:2.2) arc[start angle=227.014, end angle=180.000, radius=2.2];
\draw[gedge] (254.173:2.2) arc[start angle=254.173, end angle=227.014, radius=2.2];
\draw[gedge] (270.000:2.2) arc[start angle=270.000, end angle=254.173, radius=2.2];
\draw[gedge] (285.827:2.2) arc[start angle=285.827, end angle=270.000, radius=2.2];
\draw[gedge] (312.986:2.2) arc[start angle=312.986, end angle=285.827, radius=2.2];
\draw[gedge] (360.000:2.2) arc[start angle=360.000, end angle=312.986, radius=2.2];
\draw[gedge] (-3.700,-3.363) -- (-3.700,0.000);
\draw[gedge] (-0.000,-2.200) -- (0.000,0.000);
\draw[gedge] (-1.500,-1.609) -- (-1.500,-4.770);
\draw[gedge] (-0.600,-4.964) -- (-0.600,-2.117);
\draw[gedge] (0.600,-4.964) -- (0.600,-2.117);
\draw[gedge] (1.500,-1.609) -- (1.500,-4.770);
\draw[nstr] (360:5.18) arc[start angle=360, end angle=180, radius=5.18];
\draw[nstr] (360:2.3800000000000003) arc[start angle=360, end angle=180, radius=2.3800000000000003];
\fill (-5.000,0.000) circle (2.5pt);
\fill (-3.700,-3.363) circle (2.5pt);
\fill (-1.500,-4.770) circle (2.5pt);
\fill (-0.600,-4.964) circle (2.5pt);
\fill (0.600,-4.964) circle (2.5pt);
\fill (1.500,-4.770) circle (2.5pt);
\fill (5.000,-0.000) circle (2.5pt);
\fill (-2.200,0.000) circle (2.5pt);
\fill (-1.500,-1.609) circle (2.5pt);
\fill (-0.600,-2.117) circle (2.5pt);
\fill (-0.000,-2.200) circle (2.5pt);
\fill (0.600,-2.117) circle (2.5pt);
\fill (1.500,-1.609) circle (2.5pt);
\fill (2.200,-0.000) circle (2.5pt);
\fill (-3.700,0.000) circle (2.5pt);
\fill (0.000,0.000) circle (2.5pt);
\node[depthcolor] at (-1.0,-0.9) {$2$};
\node[depthcolor] at (1.0,-0.9) {$2$};
\node[depthcolor] at (-4.2,-1.05) {$1$};
\node[depthcolor] at (-2.65,-2.25) {$1$};
\node[depthcolor] at (-1.05,-3.3) {$1$};
\node[depthcolor] at (0,-3.65) {$1$};
\node[depthcolor] at (1.05,-3.3) {$1$};
\node[depthcolor] at (2.65,-1.95) {$1$};
\node[depthcolor] at (0,-5.45) {$0$};
\end{tikzpicture}
\caption{A web drawn in the form of Corollary~\ref{cor: Arcs and vertical edges}.}\label{fig: running example half circles}
\end{figure}
\end{example}

\begin{remark}
    Note that $L_{(1,n)}(S_T)$ is exactly the band diagram for $G_T$ defined in \cite[Section 3.1]{RTshadow}.
\end{remark}

\begin{remark}
Corollary~\ref{cor: Arcs and vertical edges} suggests that the $(1,n)$-strands can be
taken as the primary data of the web, with the remaining edges determined by how those
strands nest. In forthcoming work, Steinberg \cite{Steinberg} adopts this perspective,
constructing the webs of Fontaine's basis \cite{FON} by first building their
$(1,n)$-strands and recovering the rest of the web from them.
\end{remark}

\subsection{The tableau stranding is a leading term stranding}\label{subsection: web basis}

Finally, we show that $S_T$ is a leading term stranding of $G_T$. First, we observe that $S_T$ agrees with a leading term stranding of $G_T$ at $e_1$. This anchors the
argument in the proof of Theorem~\ref{thm: matching is lex-leading}.

\begin{lemma}\label{lem: vertex 1 lex-minimal}
Let $T\in \mathcal{RST}_n(\vec{k})$, and let $G_T$, $S_T$
be the web and stranding produced by the construction of
Section~\ref{subsection: stranded web}. Consider some leading term stranding $S^\star$ of $G_T$. Then $S^{\star}(e_1)=S_T(e_1)$. 
\end{lemma}

\begin{proof}
Since $T$ is a row-strict tableau with content $\{1^{k_1},\ldots, m^{k_m}\}$, the entry $1$ occupies rows $1,\ldots,k_1$ of $T$. By Theorem~\ref{thm: monomial of S_T}, the first tensor factor of $x_{S_T}$ is $x_1\wedge \cdots \wedge x_{k_1}$. Since this is lexicographically minimal, and $S^\star$ is a leading term stranding, $S^{\star}(e_1)=S_T(e_1)$.
\end{proof}

Next, we have two technical lemmas that drive the proof of our main theorem.

\begin{lemma}\label{lem: crossing disagreement}
    Consider some crossing of $M_T$ and its corresponding edges in $G_T$, and choose some stranding $S$ of $G_T$. If $S$ agrees with $S_T$ on all outgoing edges of this crossing, then it also agrees with $S_T$ on all incoming edges. 
\end{lemma}
\begin{proof}

Say $S$ assigns incoming edges the binary labels $\beta$ and $\beta'$ with $c$ and $c'$ nonzero entries, respectively, where $c<c'$. By Lemma~\ref{lem: flow condition on binary labels}, the vertical edge $e$ coming from the crossing has binary label $b_S(e)=\beta'-\lambda_c=\lambda_{c'}-\beta.$ This implies $b_S(e)$ has 0 in the first $c$ entries and 0 in the last $n-c'$ entries. Since $b_S(e)$ has exactly $c'-c$ ones, it follows that $b_S(e)=\lambda_{c'}-\lambda_c$ so that $\beta'=\lambda_{c'}$ and $\beta=\lambda_c$.
\end{proof}

\begin{lemma}\label{lem: trunk disagreement}
    Consider some boundary vertex $v_j$  of $G_T$ with incident edge $e_j$ and some stranding $S$ of $G_T$. Say $S$ agrees with $S_T$ on all outgoing branches of the trunk at $v_j$. Then $\hat{b}_{S_T}(e_j)\preceq \hat{b}_S(e_j)$. Moreover, $\hat{b}_{S_T}(e_j)= \hat{b}_S(e_j)$ if and only if $S_T$ and $S$ agree on all incoming branch edges at $v_j$.
\end{lemma}
\begin{proof}
First, we consider the two extremal cases. If there are no incoming branches at $v_j$, then $e_j$ is an edge of weight $h_1$ oriented out of the boundary. By assumption then, $\hat{b}_{S_T}(e_j)= \hat{b}_S(e_j)$. Now say there are no outgoing branches at $v_j$. Then $e_j$ is an edge of weight $r_1$ oriented into the boundary, and $b_{S_T}(e_j)=\lambda_{r_1}$. This is the lexicographically minimal choice, so $\hat{b}_{S_T}(e_j)\preceq \hat{b}_S(e_j)$.

For the remainder of the argument, say $p,q\geq 1$, and let $\beta_1, \ldots, \beta_q$ be the binary labels assigned by $S$ to the incoming branches colored $r_1,\ldots, r_q$, respectively. Recall that $S_T$ assigns each incoming (resp. outgoing) branch a single simple strand of color $r_i$ (resp. $h_i$), so by Theorem~\ref{thm:bijection} $S_T$ assigns binary labels $\lambda_{r_i}$
(resp. $\lambda_{h_i}$) to these branches. The binary labels for trunk edges are completely determined by the labels on the branches, by iterating Lemma~\ref{lem: flow condition on binary labels} at each vertex of the trunk. In particular, we have $$b_{S}(e_j)=\begin{cases} 
\sum_{i=1}^q \beta_i-\sum_{i=1}^p \lambda_{h_i} & \textup{if } \sigma_{v_j}(e_j)=1\\
\sum_{i=1}^p \lambda_{h_i}-\sum_{i=1}^q \beta_i & \textup{if } \sigma_{v_j}(e_j)=-1
\end{cases} \textup{ and }$$ 
$$b_{S_T}(e_j)=\begin{cases} 
\sum_{i=1}^q \lambda_{r_i}-\sum_{i=1}^p \lambda_{h_i} & \textup{if } \sigma_{v_j}(e_j)=1\\
\sum_{i=1}^p \lambda_{h_i}-\sum_{i=1}^q \lambda_{r_i} & \textup{if } \sigma_{v_j}(e_j)=-1
\end{cases}.$$
We will assume $0=r_0<r_1<\cdots$ for the remainder of the argument. The case that these values decrease with index is analogous. Observe that $$\sum_{i=1}^q \lambda_{r_i}=(q,\ldots, q,q-1,\ldots, q-1, \ldots, 1, \ldots, 1, 0, \ldots, 0)$$ has $q-t+1$ in entries $r_{t-1}+1$ through $r_t$ where $1\leq t\leq q$ and 0 elsewhere. 

Say $S$ and $S_T$ disagree at $e_j$. Because both $b_S(e_j)$ and $b_{S_T}(e_j)$ are binary vectors, each of their entries differs by at most 1. Consider the first entry $c$ where they differ, and note that this implies the $c$th entries of $\sum_{i=1}^{q} \beta_i$ and $\sum_{i=1}^{q} \lambda_{r_i}$ differ by 1. Say $r_{t-1}< c\leq r_{t}$. 
Since $\beta_i$ is a binary vector with exactly $r_i$ nonzero entries, it follows that $\sum_{i=1}^q \beta_i$ has at most $r_t$ entries larger than $q-t$. Since $\sum_{i=1}^{q} \beta_i$ and $\sum_{i=1}^{q} \lambda_{r_i}$ agree prior to position $c$, the first $r_{t-1}$ entries of $\sum_{i=1}^{q} \beta_i$ are larger than $q-t+1$. The $c$th entry of $\sum_{i=1}^{q} \lambda_{r_i}$ is $q-t+1$. Since $\sum_{i=1}^{q} \beta_i$ cannot have another entry larger than $q-t+1$, we conclude the $c$th entry of $\sum_{i=1}^{q} \beta_i$ must be $q-t$. 

Since the $c$th entry of $\sum_{i=1}^{q} \beta_i$ is smaller than that of $\sum_{i=1}^{q} \lambda_{r_i}$, we get that when $\sigma_{v_j}(e_j)=1$, the $c$th entry of $b_S(e_j)=\sum_{i=1}^q \beta_i-\sum_{i=1}^{p} \lambda_{h_i}$ is 0 and the $c$th entry of $b_{S_T}(e_j)=\sum_{i=1}^q \lambda_{r_i}-\sum_{i=1}^{p} \lambda_{h_i}$ is 1. On the other hand, if $\sigma_{v_j}(e_j)=-1$, we have the $c$th entry of $b_S(e_j)=\sum_{i=1}^{p} \lambda_{h_i}-\sum_{i=1}^q \beta_i$ is 1 and the $c$th entry of $b_{S_T}(e_j)=\sum_{i=1}^{p} \lambda_{h_i}-\sum_{i=1}^q \lambda_{r_i}$ is 0. When $\sigma_{v_j}(e_j)=1$, we have $\hat{b}_S(e_j)=b_{S}(e_j)$ and $\hat{b}_{S_T}(e_j)=b_{S_T}(e_j)$. When $\sigma_{v_j}(e_j)=-1$, we have $\hat{b}_S(e_j)=\vec{1}-b_{S}(e_j)$ and $\hat{b}_{S_T}(e_j)=\vec{1}-b_{S_T}(e_j)$. Therefore we get $\hat{b}_{S_T}(e_j)\prec \hat{b}_S(e_j)$ for both values of $\sigma_{v_j}(e_j)$.

Now say $S$ and $S_T$ agree at $e_j$ so that $\sum_{i=1}^q \beta_i=\sum_{i=1}^q \lambda_{r_i}$. Since $\sum_{i=1}^q \lambda_{r_i}$ has $q$ in the first $r_1$ entries, so does $\sum_{i=1}^q \beta_i$. Thus, all vectors in the set $\{\beta_1, \ldots, \beta_q\}$ have one in the first $r_1$ entries. Since $\beta_1$ has exactly $r_1$ nonzero entries, we conclude $\beta_1=\lambda_{r_1}$. Continuing inductively, there are $q-t+1$ vectors in the set $\{\lambda_{r_1}, \ldots \lambda_{r_{t-1}},\beta_t, \ldots, \beta_q\}$ with 1 in the first $r_t$ entries.  Since $r_1<\cdots <r_{t-1}<r_t$, it must be vectors $\{\beta_t, \ldots, \beta_q\}$ with 1 in the first $r_t$ entries. Since $\beta_t$ has exactly $r_t$ nonzero entries, $\beta_t=\lambda_{r_t}$.
\end{proof}

\begin{theorem}\label{thm: matching is lex-leading}
Let $T\in\mathcal{RST}_n(\vec{k})$, $M_T$ be any colored noncrossing
matching for $T$, and $G_T$ and $S_T$ be the web and stranding produced
by Construction~\ref{def: G_T and S_T}. Then $S_T$ is a leading
term stranding of $G_T$.
\end{theorem}

\begin{proof}

By Corollary \ref{cor: Arcs and vertical edges}, we may assume the $(1,n)$-strands of $S_T$ are clockwise open strands drawn as half circles, and the remaining edges of $G_T$ are vertical, connecting nested pairs of these strands. 

    Let $S^{\star}$ be a leading term stranding of $G_T$, and say $x_{S_T}\neq x_{S^\star}$. Let $e_1,\ldots, e_m$ be the boundary edges of $G_T$, and let $e_j$ be the first edge where $\hat{b}_{S_T}(e_j)\neq \hat{b}_{S^{\star}}(e_j)$.  By Lemma \ref{lem: vertex 1 lex-minimal}, we know $j>1$. First, we argue that $S_T$ and $S^{\star}$ disagree on some outgoing branch of the trunk at $v_j$.

    There must be outgoing branches on the trunk at $v_j$ since otherwise $\hat{b}_{S_T}(e_j)$ is lexicographically minimal. If there are no incoming branches, $e_j$ is an outgoing branch on which $S_T$ and $S^{\star}$ disagree. Now say there are both incoming and outgoing branches on the trunk at $v_j$, and let $h_1, \ldots, h_p$ and $r_1,\ldots, r_q$ be the colors of outgoing and incoming branches, respectively. Since $S^{\star}$ is a leading term stranding, we have $\hat{b}_{S^{\star}}(e_j)\prec \hat{b}_{S_T}(e_j)$. By Lemma \ref{lem: trunk disagreement}, $S^{\star}$ and $S_T$ must disagree on some outgoing branch of the trunk at $v_j$. Denote this edge by $f_1$. Since branches separate faces of different depth, by Lemma~\ref{lem: (1,n) strands}, $f_1$ lies on some $(1,n)$-strand of $S_T$. 
    
    Because $S_T$ and $S^{\star}$ agree at boundary vertices to the left of $v_j$, the head of $f_1$ is not incident to the boundary, so  $f_1$ has its head at an interior vertex endpoint of a vertical trunk or crossing edge of $G_T$. By Lemmas \ref{lem: trunk disagreement} and \ref{lem: crossing disagreement}, there is an outgoing edge $f_{2}$ of this trunk or crossing on which $S^\star$ and $S_T$ disagree. Since $f_{2}$ is an outgoing edge, it is part of $L_{(1,n)}(S_T)$. Moreover, its head is to the left of the head of $f_1$. Continuing in this way, we obtain an unending sequence $f_1, f_2, \ldots$ of distinct edges of $L_{(1,n)}(S_T)$ on which $S_T$ and $S^{\star}$ disagree. 

    Since $G_T$ is a finite graph this is not possible. We conclude $x_{S^{\star}}=x_{S_T}$, and $S_T$ is a leading term stranding of $G_T$ as desired.
\end{proof}

Theorems~\ref{thm: monomial of S_T} and~\ref{thm: matching is lex-leading} together
show that $S_T$ is a leading term stranding of $G_T$ with $T_{S_T}=T$. As a result, we have $|\{G_T: T\in \mathcal{RST}_n(\vec{k})\}|=|\mathcal{RST}_n(\vec{k})|$, and
Corollary~\ref{cor: basis criterion via tableaux} gives the following.

\begin{corollary}\label{cor: matching produces basis}
For each $T\in\mathcal{RST}_n(\vec{k})$, fix a colored noncrossing matching $M_T$ and let $G_T$ be the resulting web. Then the set $\{G_T : T\in \mathcal{RST}_n(\vec{k})\}$ is a web basis for $\textup{Inv}_n(\vec{k})$. That is, $\{f^{\uparrow}(G_T) : T\in \mathcal{RST}_n(\vec{k})\}$ is a basis of $\textup{Inv}_n(\vec{k})$.
\end{corollary}

To end this section, we include an example of distinct web vectors obtained from the same tableau via our algorithm. 

\begin{example}\label{ex: same leading term}
Let $\vec{k}=(2,3,1,2,1,3)$, and let $T\in\mathcal{RST}_4(\vec{k})$ be
\[
T=\ytableaushort{123,126,246,456}.
\]
The pairing of Construction~\ref{def:tableau matching} has arcs $(v_1, v_4)$ in color $2$,
$(v_2,v_5)$ in color $3$, and $(v_3, v_6)$ in color $1$. The three arcs pairwise cross,
and such a diagram admits two isotopy classes as shown in Figure~\ref{fig: same leading term}. Write $M_T$ and $M'_T$ for the two matchings and $G_T$, $G'_T$
for the webs obtained from them by Construction~\ref{def: G_T and S_T}. The recursive construction of \cite{FON} returns
$G'_T$.

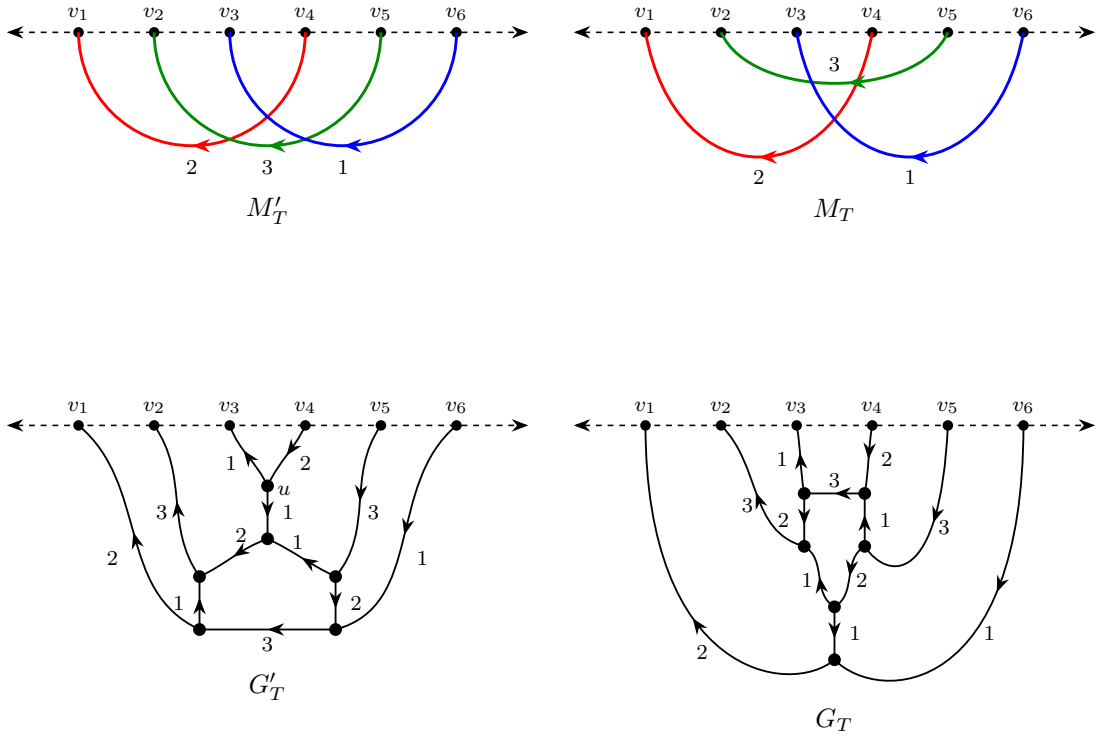
\begin{figure}[htbp]
\centering
\begin{tikzpicture}[
  mtarc/.style={line width=1.1pt, decoration={markings,
      mark=at position #1 with {\arrow{Stealth[length=2.4mm]}}},
      postaction={decorate}},
  mtarc/.default=0.5,
  gedge/.style={line width=0.7pt, decoration={markings,
      mark=at position #1 with {\arrow{Stealth[length=2.2mm]}}},
      postaction={decorate}},
  gedge/.default=0.55,
  bpt/.style={fill,circle,inner sep=1.4pt},
  ivt/.style={fill,circle,inner sep=1.7pt},
  wt/.style={font=\scriptsize},
  scale=1,
]

\def\bdry{%
  \draw[line width=0.6pt, dashed, dash pattern=on 2.2pt off 2.2pt,
        {Stealth[length=2.2mm]}-{Stealth[length=2.2mm]}] (-0.95,0) -- (5.95,0);
  \foreach \x/\n in {0/1,1/2,2/3,3/4,4/5,5/6} {
    \node[bpt] at (\x,0) {};
    \node[font=\scriptsize,above] at (\x,0.02) {$v_{\n}$};
  }
}

\begin{scope}[shift={(0,0)}]
  \bdry
  \draw[mtarc=0.5,red]            (3,0) arc (0:-180:1.5);
  \draw[mtarc=0.5,green!55!black] (4,0) arc (0:-180:1.5);
  \draw[mtarc=0.5,blue]           (5,0) arc (0:-180:1.5);
  \node[wt] at (1.50,-1.78) {$2$};
  \node[wt] at (2.50,-1.78) {$3$};
  \node[wt] at (3.50,-1.78) {$1$};
  \node[font=\small] at (2.5,-2.35) {$M'_T$};
\end{scope}

\begin{scope}[shift={(7.5,0)}]
  \bdry
  \draw[mtarc=0.5,red]             (3,0) .. controls (2.6,-2.2) and (0.4,-2.2) .. (0,0);
  \draw[mtarc=0.45,green!55!black] (4,0) .. controls (3.6,-0.9) and (1.4,-0.9) .. (1,0);
  \draw[mtarc=0.5,blue]            (5,0) .. controls (4.6,-2.2) and (2.4,-2.2) .. (2,0);
  \node[wt] at (1.50,-1.92) {$2$};
  \node[wt] at (2.50,-0.42) {$3$};
  \node[wt] at (3.50,-1.92) {$1$};
  \node[font=\small] at (2.5,-2.35) {$M_T$};
\end{scope}

\begin{scope}[shift={(0,-5.2)}]
  \bdry
  \coordinate (UAC) at (2.5,-0.80);  \coordinate (LAC) at (2.5,-1.50);
  \coordinate (UAB) at (1.6,-2.00);  \coordinate (LAB) at (1.6,-2.70);
  \coordinate (UBC) at (3.4,-2.00);  \coordinate (LBC) at (3.4,-2.70);

  \draw[gedge=0.6] (UAC) to[out=115,in=-70] (2,0);      \node[wt] at (2.02,-0.50) {$1$};
  \draw[gedge=0.5] (3,0)  to[out=-115,in=65] (UAC);     \node[wt] at (3.00,-0.50) {$2$};
  \draw[gedge=0.5] (4,0)  to[out=-125,in=50] (UBC);     \node[wt] at (3.90,-1.15) {$3$};
  \draw[gedge=0.5] (5,0)  to[out=-140,in=15] (LBC);     \node[wt] at (4.55,-1.75) {$1$};
  \draw[gedge=0.5] (UAB) to[out=125,in=-55] (1,0);      \node[wt] at (1.10,-1.15) {$3$};
  \draw[gedge=0.5] (LAB) to[out=155,in=-40] (0,0);      \node[wt] at (0.45,-1.75) {$2$};

  \draw[gedge=0.5] (LAC) to[out=200,in=25] (UAB);       \node[wt] at (2.15,-1.45) {$2$};
  \draw[gedge=0.5] (UBC) to[out=155,in=-25] (LAC);      \node[wt] at (2.90,-1.55) {$1$};
  \draw[gedge=0.5] (LBC) -- (LAB);                      \node[wt] at (2.50,-2.90) {$3$};

  \draw[gedge=0.6] (UAC) -- (LAC);                      \node[wt,right] at (2.56,-1.15) {$1$};
  \draw[gedge=0.6] (LAB) -- (UAB);                      \node[wt,left]  at (1.54,-2.35) {$1$};
  \draw[gedge=0.6] (UBC) -- (LBC);                      \node[wt,right] at (3.46,-2.35) {$2$};

  \foreach \p in {UAC,LAC,UAB,LAB,UBC,LBC} {\node[ivt] at (\p) {};}
  \node[font=\scriptsize] at (2.72,-0.86) {$u$};
  \node[font=\small] at (2.5,-3.45) {$G'_T$};
\end{scope}

\begin{scope}[shift={(7.5,-5.2)}]
  \bdry
  \coordinate (uBC) at (2.10,-0.90);  \coordinate (lBC) at (2.10,-1.60);
  \coordinate (uAB) at (2.90,-0.90);  \coordinate (lAB) at (2.90,-1.60);
  \coordinate (uAC) at (2.50,-2.40);  \coordinate (lAC) at (2.50,-3.10);

  \draw[gedge=0.6] (uBC) to[out=100,in=-80] (2,0);      \node[wt,left]  at (2.02,-0.45) {$1$};
  \draw[gedge=0.5] (3,0)  to[out=-100,in=80] (uAB);     \node[wt,right] at (2.98,-0.45) {$2$};
  \draw[gedge=0.5] (uAB) -- (uBC);                      \node[wt] at (2.50,-0.72) {$3$};
  \draw[gedge=0.5] (lBC) to[out=165,in=-45] (1,0);      \node[wt] at (1.35,-1.05) {$3$};
  \draw[gedge=0.5]  (4,0) .. controls (3.9,-2.15) and (3.2,-2.05) .. (lAB);
        \node[wt] at (3.95,-1.35) {$3$};
  \draw[gedge=0.45] (5,0) .. controls (5.0,-3.6) and (3.1,-3.7) .. (lAC);
        \node[wt] at (4.55,-2.75) {$1$};
  \draw[gedge=0.45] (lAC) .. controls (1.9,-3.6) and (0,-3.3) .. (0,0);
        \node[wt] at (0.75,-3.00) {$2$};

  \draw[gedge=0.5] (lAB) to[out=210,in=30] (uAC);       \node[wt] at (2.86,-2.05) {$2$};
  \draw[gedge=0.5] (uAC) to[out=150,in=-30] (lBC);      \node[wt] at (2.14,-2.05) {$1$};

  \draw[gedge=0.6] (uBC) -- (lBC);                      \node[wt,left]  at (2.04,-1.25) {$2$};
  \draw[gedge=0.6] (lAB) -- (uAB);                      \node[wt,right] at (2.96,-1.25) {$1$};
  \draw[gedge=0.6] (uAC) -- (lAC);                      \node[wt,right] at (2.56,-2.75) {$1$};

  \foreach \p in {uBC,lBC,uAB,lAB,uAC,lAC} {\node[ivt] at (\p) {};}
  \node[font=\small] at (2.5,-3.90) {$G_T$};
\end{scope}

\end{tikzpicture}
\caption{Two colored noncrossing matchings for the tableau $T$ of
Example~\ref{ex: same leading term} and the webs they produce. Blue, red, and green represent colors 1, 2, and 3, respectively.}
\label{fig: same leading term}
\end{figure}

By Theorem~\ref{thm: matching is lex-leading}, $G_T$ and $G'_T$ have web vectors with the same leading term. However, the vectors themselves differ. Consider the
boundary data with $\hat b$-values $1010,\,0111,\,1000,\,1001,\,0100,\,0111$. One can verify that $G_T$ supports a stranding with this boundary data, so the monomial \[
x_1\wedge x_3\otimes x_2\wedge x_3\wedge x_4\otimes x_1\otimes
x_1\wedge x_4\otimes x_2\otimes x_2\wedge x_3\wedge x_4 
\] occurs in
$f^{\uparrow}(G_T)$ with nonzero coefficient by
Theorem~\ref{thm: nonzero coefficient}. On the other hand, the boundary
edges of $G'_T$ at $v_3$ and $v_4$ meet at the vertex $u$ of
Figure~\ref{fig: same leading term}, and there is no way to validly label the third edge at $u$ compatible with the given boundary data. Hence the monomial
does not occur in $f^{\uparrow}(G'_T)$, and $f^{\uparrow}(G_T)\neq f^{\uparrow}(G'_T)$.
\end{example}

\section{Leading term strandings for \texorpdfstring{$\mathfrak{sl}_3$}{sl3} webs}\label{section: sl3 stranding}

In the preceding section, we focused on constructing $\mathfrak{sl}_n$ webs with specific leading terms. Here, we focus instead on finding the leading term for an arbitrary $\mathfrak{sl}_3$ web. To simplify our discussion of $\mathfrak{sl}_3$ webs, we will make use of the {\bf edge flip relation} from the kernel of $f^{\uparrow}$ defined as follows.

\begin{definition}[Edge flip]\label{def: edge flip}
For an edge $e=u\stackrel{\ell}{\mapsto}v$ of an $\mathfrak{sl}_n$ web $G$, the {\bf edge flip} of $e$ is $\varphi(e)=v\stackrel{n-\ell}{\longmapsto}u$. For $\mathcal{E}\subseteq E(G)$, write $G_{\varphi(\mathcal{E})}$ for the web obtained from $G$ by flipping every edge in $\mathcal{E}$.
\end{definition}

\begin{lemma}[{\cite[Lemma~44]{RTStranding}}]\label{lem: edge flip relation}
For any $G\in F_n(\vec{k})$ and any $\mathcal{E}\subseteq E(G)$, $f^{\uparrow}(G)=f^{\uparrow}(G_{\varphi(\mathcal{E})})$.
\end{lemma}

By the edge flip relation, we may assume that all $\mathfrak{sl}_3$ web edges have weight $1$. For such webs, the congruence condition on edge weights at interior vertices forces every interior vertex to be a source or sink. We adopt these assumptions throughout this section.

\subsection{Leading term strandings for simple \texorpdfstring{$\mathfrak{sl}_3$}{sl3} webs}

Say $e$ is an edge of some $\mathfrak{sl}_3$ web $G$ with face $F$ to its left and $F'$ to its right. Given $S\in \mathcal{S}tr(G)$, there are three possibilities:
\begin{enumerate}
    \item $b_S(e)=100$, $S^+(e)=\{1\}$, $S^-(e)=\emptyset$, and $d_S(F)<d_S(F')$;
    \item $b_S(e)=010$; $S^+(e)=\{2\}$, $S^-(e)=\{1\}$, and $d_S(F)=d_S(F')$; or
    \item $b_S(e)=001$, $S^+(e)=\emptyset$, $S^-(e)=\{2\}$, and $d_S(F)>d_S(F')$.
\end{enumerate}
  At the same time, integral depth can behave in exactly three ways on either side of $e$: $d(F) < d(F')$, $d(F) = d(F')$, or $d(F) > d(F')$. 
  
  The main result of this subsection is that these two trichotomies align perfectly for so-called simple webs: for such webs, there is a unique stranding $S_{\mathrm{dep}}$ satisfying $d_{S_{\mathrm{dep}}}(F) = d(F)$ for every face $F$, and $S_{\mathrm{dep}}$ is a leading term stranding.

\begin{definition}[Flat vertex and simple webs]\label{def: simple web}
Let $G$ be a half-plane graph with interior vertex $v$.
\begin{itemize}
    \item We call $v$ a \textbf{flat vertex} if all integral depths of faces adjacent to $v$ are the same.
    \item The web $G$ is \textbf{simple} if it has no flat vertices.
\end{itemize}
\end{definition}

\begin{theorem}\label{thm:sl3 stranding}
Let $G$ be a simple $\mathfrak{sl}_3$ web.
\begin{enumerate}
\item[(i)] The edge-wise stranding choices illustrated in Figure~\ref{fig:sl3-edge-labels} constitute a valid stranding $S_{\mathrm{dep}}$ of $G$ satisfying $d_{S_{\mathrm{dep}}}(F) = d(F)$ for every face $F$ of $G$, and $S_{\mathrm{dep}}$ is the unique stranding of $G$ with this property.
\item[(ii)] The stranding $S_{\mathrm{dep}}$ is a leading term stranding for $G$.
\end{enumerate}
\end{theorem}

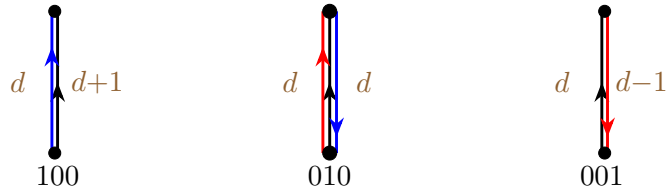
\begin{figure}[h]
\centering
\begin{tikzpicture}[scale=.75,
  edge/.style={line width=1.1pt, decoration={markings, mark=at position 0.5 with {\arrow{>}}}, postaction={decorate}},
  strand/.style={decoration={markings, mark=at position 0.75 with {\arrow{>}}}, postaction={decorate}},
  strandrev/.style={decoration={markings, mark=at position 0.25 with {\arrow{<}}}, postaction={decorate}}
]
\begin{scope}[xshift=0cm]
  \draw[black, edge] (0, 0) -- (0, 2.5);
  \draw[blue, line width=1.1pt, strand] (-0.1, 0) -- (-0.1, 2.5);
  \node[depthcolor] at (-.7, 1.25) {$d$};
  \node[depthcolor] at (.7,  1.25) {$d{+}1$};
  \node at (0, -0.4) {$100$};
  \filldraw[black] (-.05,0) circle (3pt);
  \filldraw[black] (-.05,2.5) circle (3pt);
\end{scope}
\begin{scope}[xshift=4.8cm]
  \draw[black, edge] (0, 0) -- (0, 2.5);
  \draw[red, line width=1.1pt, strand] (-0.12, 0) -- (-0.12, 2.5);
  \draw[blue, line width=1.1pt, strandrev] (0.12, 0) -- (0.12, 2.5);
  \node[depthcolor] at (-.7, 1.25) {$d$};
  \node[depthcolor] at (.6,  1.25) {$d$};
  \node at (0, -0.4) {$010$};
  \filldraw[black] (0,0) circle (3.5pt);
  \filldraw[black] (0,2.5) circle (3.5pt);
\end{scope}
\begin{scope}[xshift=9.6cm]
  \draw[black, edge] (0, 0) -- (0, 2.5);
  \draw[red, line width=1.1pt, strandrev] (0.1, 0) -- (0.1, 2.5);
  \node[depthcolor] at (-.7, 1.25) {$d$};
  \node[depthcolor] at (.7,  1.25) {$d{-}1$};
  \node at (0, -0.4) {$001$};
  \filldraw[black] (0.05,0) circle (3pt);
  \filldraw[black] (0.05,2.5) circle (3pt);
\end{scope}
\end{tikzpicture}
\caption{Unique edge stranding choices that align integral and strand depth in an $\mathfrak{sl}_3$ web along with corresponding binary labels. Depths are labeled in brown; blue and red strands carry colors $1$ and $2$, respectively.}\label{fig:sl3-edge-labels}
\end{figure}

\begin{proof}
We first prove (i). Once we show $S_{\mathrm{dep}}$ is valid, it follows by induction with base step $d_{S_{\mathrm{dep}}}(U)=d(U)=0$ that $d_{S_{\mathrm{dep}}}(F)=d(F)$ on all faces $F$ of $G$. The stranding choices illustrated in Figure~\ref{fig:sl3-edge-labels} are valid on individual edges of $G$. Therefore, we need only check these choices are compatible at interior vertices. Let $v$ be an interior vertex of $G$. Since $G$ is simple and depth on either side of an edge can differ by at most one, exactly two distinct depths occur among the three faces adjacent to $v$. Therefore, at $v$ integral depth increases across one incident edge, decreases across another, and stays the same across the third. This implies each edge incident to $v$ has a different stranding pattern. Since $v$ is either a source or sink, these edge strandings meet compatibly at $v$, so $S_{\mathrm{dep}}$ is a valid stranding of $G$. For uniqueness, let $S\in\mathcal{S}tr(G)$ satisfy $d_S(F)=d(F)$ for every face $F$. By the trichotomy preceding Definition~\ref{def: simple web}, the label $b_S(e)$ on each edge $e$ is determined by whether strand depth increases, stays the same, or decreases across $e$. Since $d_S=d=d_{S_{\mathrm{dep}}}$ on every face, $b_S(e)=b_{S_{\mathrm{dep}}}(e)$ for every edge, so $S=S_{\mathrm{dep}}$.

We now prove (ii). Denote by $v_1, \ldots, v_m$ the boundary vertices of $G$ read left to right, with corresponding boundary edges $e_1, \ldots, e_m$ and boundary faces $F_0,\ldots, F_m$ where $F_0=F_m=U$ and $F_i$ is incident to the boundary axis between edges $e_i$ and $e_{i+1}$ for $1\leq i <m$. Let $S^\star$ be a leading term stranding for $G$. We show $S_{\mathrm{dep}}(e_i)=S^\star(e_i)$ for all $1\leq i\leq m$, so that $S_{\mathrm{dep}}$ is also a leading term stranding.

We have $d(F_0)=d(U)=0$ and $d(F_1)=1$. If $\sigma_{v_1}(e_1)=1$, then $b_{S_{\mathrm{dep}}}(e_1)=\hat{b}_{S_{\mathrm{dep}}}(e_1)=100$. If $\sigma_{v_1}(e_1)=-1$, then $b_{S_{\mathrm{dep}}}(e_1)=001$ and $\hat{b}_{S_{\mathrm{dep}}}(e_1)=110$. In either case,  $\hat{b}_{S_{\mathrm{dep}}}(e_1)$ is lexicographically minimal, so $S_{\mathrm{dep}}(e_1)=S^\star(e_1)$. Since $d_{S_{\mathrm{dep}}}(F) = d(F)$ for all $F$, we also have $d_{S^\star}(F_i)=d_{S_{\mathrm{dep}}}(F_i)=d(F_i)$ for $i=0,1$.

Now let $1<i<m$, and suppose $S_{\mathrm{dep}}(e_j)=S^\star(e_j)$ and $d_{S^\star}(F_j)=d_{S_{\mathrm{dep}}}(F_j)=d(F_j)$ for all $1\leq j<i$.
\begin{itemize}
    \item If $d(F_i)=d(F_{i-1})+1$, then $b_{S_{\mathrm{dep}}}(e_i)$ is lexicographically minimal by the same argument as above, so $S_{\mathrm{dep}}(e_i)=S^\star(e_i)$. Moreover, $d_{S^\star}(F_i)=d_{S^\star}(F_{i-1})+1=d_{S_{\mathrm{dep}}}(F_{i-1})+1=d_{S_{\mathrm{dep}}}(F_i)$.
    \item If $d(F_i)=d(F_{i-1})$, then by Lemma~\ref{lem: depth bound}, $d_{S^\star}(F_i)\leq d(F_i)=d(F_{i-1})=d_{S^\star}(F_{i-1})$, so the number of simple strands of $S^\star$ enclosing $F_i$ is at most the number enclosing $F_{i-1}$. By Lemma \ref{lem: open strands clockwise}, all open strands of $S^{\star}$ are clockwise, so $S^\star$ cannot assign $e_i$ a single strand oriented into the boundary. If $\sigma_{v_i}(e_i)=1$, this means $100\neq\hat{b}_{S^\star}(e_i)$ so that $\hat{b}_{S^\star}(e_i)\in\{010,001\}$. Since $\hat{b}_{S_{\mathrm{dep}}}(e_i)=010\prec 001$ and $S^\star$ is a leading term stranding, $\hat{b}_{S^\star}(e_i)=\hat{b}_{S_{\mathrm{dep}}}(e_i)=010$. If $\sigma_{v_i}(e_i)=-1$, a parallel argument shows $\hat{b}_{S^\star}(e_i)=\hat{b}_{S_{\mathrm{dep}}}(e_i)=101$. Moreover, $d_{S^\star}(F_i)=d_{S^\star}(F_{i-1})=d_{S_{\mathrm{dep}}}(F_{i-1})=d_{S_{\mathrm{dep}}}(F_i)$.
    \item If $d(F_i)=d(F_{i-1})-1$, then Lemma~\ref{lem: depth bound} implies that strand depth in both $S^\star$ and $S_{\mathrm{dep}}$ must be smaller on $F_i$ than $F_{i-1}$. This means both $S^\star$ and $S_{\mathrm{dep}}$ assign a single strand oriented out of the boundary to $e_i$, and $d_{S^\star}(F_i)=d_{S^\star}(F_{i-1})-1=d_{S_{\mathrm{dep}}}(F_{i-1})-1=d_{S_{\mathrm{dep}}}(F_i)$.
\end{itemize}

This proves $S_{\mathrm{dep}}(e_i)=S^\star(e_i)$ for all $1\leq i\leq m$, so $x_{S^\star}=x_{S_{\mathrm{dep}}}$ and $S_{\mathrm{dep}}$ is a leading term stranding.
\end{proof}

\begin{remark}\label{rem: flat vertex obstruction}
At a flat vertex of a web, all three adjacent faces have the same integral depth, so the stranding rule illustrated in Figure~\ref{fig:sl3-edge-labels} would place the same pair of strands on all three incident edges. This is not a valid stranding, so the algorithm does not extend beyond simple webs.
\end{remark}

\begin{example}\label{ex:no saturated stranding}
Figure~\ref{fig:no saturated stranding} shows two $\mathfrak{sl}_3$ webs with all edges of weight $1$ and faces labeled by integral depth. The web on the top left is simple; the one on the bottom left has a flat vertex. On the top right, we see the leading term stranding $S_{\mathrm{dep}}$ from Theorem~\ref{thm:sl3 stranding}. On the bottom right, we demonstrate that the algorithm fails at the flat vertex.
\begin{figure}[h]
\centering
\raisebox{10pt}{\begin{tikzpicture}[scale=.75]
\draw[axis, <->] (0,0)--(7,0);
\begin{scope}[line width=0.7pt,decoration={markings, mark=at position 0.5 with {\arrow{<}}}]
\draw[postaction={decorate},line width=0.7pt] (1,0) to[out=270,in=180] (3,-2);
\draw[postaction={decorate},line width=0.7pt] (4,-2)--(4,-1);
\draw[postaction={decorate},line width=0.7pt] (5,0) to[out=270,in=0] (4,-1);
\draw[postaction={decorate},line width=0.7pt] (3,-1) to[out=180,in=270] (2,0);
\end{scope}
\begin{scope}[line width=0.7pt,decoration={markings, mark=at position 0.5 with {\arrow{>}}}]
    \draw[postaction={decorate},line width=0.7pt] (4,-1)--(3,-1);
    \draw[postaction={decorate},line width=0.7pt] (3,-2)--(3,-1);
    \draw[postaction={decorate},line width=0.7pt] (3,-2)--(4,-2);
\draw[postaction={decorate},line width=0.7pt] (6,0) to[out=270,in=0] (4,-2);
\end{scope}
\fill(1,0) circle (2.5pt);
\fill(3,-1) circle (2.5pt);
\fill(4,-1) circle (2.5pt);
\fill(3,-2) circle (2.5pt);
\fill(4,-2) circle (2.5pt);
\fill(2,0) circle (2.5pt);
\fill(5,0) circle (2.5pt);
\fill(6,0) circle (2.5pt);
\node at (1.75, -1) {\textcolor{depthcolor}{$1$}};
\node at (3.5, -1.5) {\textcolor{depthcolor}{$1$}};
\node at (5.25, -1) {\textcolor{depthcolor}{$1$}};
\node at (3.5, -.5) {\textcolor{depthcolor}{$2$}};
\end{tikzpicture}}
\hspace{.25in}
\raisebox{10pt}{\begin{tikzpicture}[scale=.75]
\draw[axis, <->] (0,0)--(7,0);
\begin{scope}[line width=1.1pt,decoration={markings, mark=at position 0.6 with {\arrow{<}}}]
\draw[postaction={decorate},line width=1.1pt, blue] (1,0) to[out=270,in=180] (3,-2);
\draw[postaction={decorate},line width=1.1pt, blue] (3.05,-2)--(3.05,-1);
\draw[postaction={decorate},line width=1.1pt, red] (3,-2)--(4,-2);
\draw[postaction={decorate},line width=1.1pt, red] (3.95,-2)--(3.95,-1);
\end{scope}
\begin{scope}[line width=1.1pt,decoration={markings, mark=at position 0.4 with {\arrow{>}}}]
\draw[postaction={decorate},line width=1.1pt, blue] (4.05,-2)--(4.05,-1);
\draw[postaction={decorate},line width=1.1pt, red] (5,0) to[out=270,in=0] (4,-1);
\draw[postaction={decorate},line width=1.1pt, blue] (6,0) to[out=270,in=0] (4,-2);
\draw[postaction={decorate},line width=1.1pt, red] (2.95,-2)--(2.95,-1);
\draw[postaction={decorate},line width=1.1pt, blue] (4,-1)--(3,-1);
\draw[postaction={decorate},line width=1.1pt, red] (3,-1) to[out=180,in=270] (2,0);
\end{scope}
\fill(1,0) circle (2.5pt);
\fill(3,-1) circle (2.5pt);
\fill(4,-1) circle (2.5pt);
\fill(3,-2) circle (2.5pt);
\fill(4,-2) circle (2.5pt);
\fill(2,0) circle (2.5pt);
\fill(5,0) circle (2.5pt);
\fill(6,0) circle (2.5pt);
\end{tikzpicture}}

\begin{tikzpicture}[scale=.9, yscale=-1]
\draw[axis, <->] (.5,0)--(5.5,0);
\begin{scope}[line width=0.7pt,decoration={markings,mark=at position 0.5 with {\arrow{<}}}]
\draw[postaction={decorate}] (1,0) to[out=90, in=180] (3,4);
\draw[postaction={decorate}] (5,0) to[out=90, in=0] (3,4);
\draw[postaction={decorate}] (3,3)--(3,4);
\draw[postaction={decorate}] (2,0) to[out=90, in=180] (2.5,2);
\draw[postaction={decorate}] (3,1)--  (2.5,2);
\draw[postaction={decorate}] (3,3)--  (2.5,2);
\draw[postaction={decorate}] (4,0) to[out=90, in=0] (3.5,2);
\draw[postaction={decorate}] (3,1)--  (3.5,2);
\draw[postaction={decorate}] (3,3)--  (3.5,2);
\draw[postaction={decorate}] (3,1)--(3,0);
\end{scope}
\foreach \x in {1,...,5}
  \fill(\x,0) circle (2.5pt);
\fill(3,1) circle (2.5pt);
\fill(2.5, 2) circle (2.5pt);
\fill(3.5, 2) circle (2.5pt);
\fill(3,3) circle (2.5pt);
\fill(3,4) circle (2.5pt);
\node at (2,2.5) {\textcolor{depthcolor}{$1$}};
\node at (4,2.5) {\textcolor{depthcolor}{$1$}};
\node at (3,2) {\textcolor{depthcolor}{$2$}};
\node at (2.5,.5) {\textcolor{depthcolor}{$2$}};
\node at (3.5,.5) {\textcolor{depthcolor}{$2$}};
\end{tikzpicture}
\hspace{.5in}
\begin{tikzpicture}[scale=.9, yscale=-1]
\draw[axis, <->] (.5,0)--(5.5,0);
\begin{scope}[line width=1.1pt,decoration={markings,mark=at position 0.5 with {\arrow{<}}}]
\draw[blue, postaction={decorate}] (1,0) to[out=90, in=180] (3,4);
\draw[red, postaction={decorate}] (3,4) to[out=0, in=90] (5,0);
\draw[blue,  postaction={decorate}] (2.95,4)--(2.95,3);
\draw[red,  postaction={decorate}] (3.05,3)--(3.05,4);
\draw[blue, postaction={decorate}] (2,0) to[out=90, in=180] (2.5,2);
\draw[red, postaction={decorate}] (3.05,1)--  (2.55,2);
\draw[blue, postaction={decorate}] (2.45,2)--  (2.95,1);
\draw[red, postaction={decorate}] (2.5,2)--(3,3);
\draw[red, postaction={decorate}] (3.5,2) to[out=0, in=90] (4,0);
\draw[red, postaction={decorate}] (3.05,1)--  (3.55,2);
\draw[blue, postaction={decorate}] (3.45,2)--  (2.95,1);
\draw[blue, postaction={decorate}] (3,3)--  (3.5,2);
\draw[red, postaction={decorate}] (3.05,1)--(3.05,0);
\draw[blue, postaction={decorate}] (2.95,0)--(2.95,1);
\end{scope}
\foreach \x in {1,...,5}
  \fill(\x,0) circle (2.5pt);
\fill(3,1) circle (2.5pt);
\fill(2.5, 2) circle (2.5pt);
\fill(3.5, 2) circle (2.5pt);
\fill(3,3) circle (2.5pt);
\fill(3,4) circle (2.5pt);
\end{tikzpicture}
\caption{Left: a simple $\mathfrak{sl}_3$ web (top) and an $\mathfrak{sl}_3$ web with a flat vertex (bottom). Right: the leading term stranding $S_{\mathrm{dep}}$ from Theorem~\ref{thm:sl3 stranding} applied to the simple web (top), and the failure of the algorithm at the flat vertex (bottom).}\label{fig:no saturated stranding}
\end{figure}
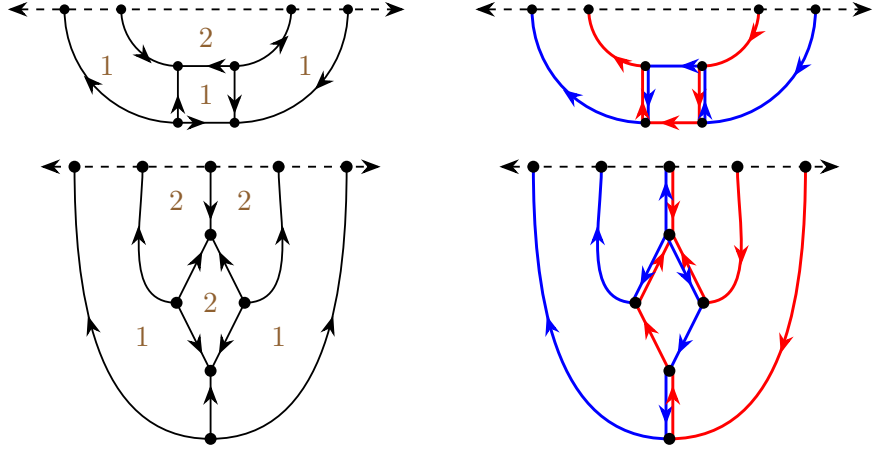
\end{example}

\begin{remark}\label{rem: sln obstruction}
Our stranding construction for simple $\mathfrak{sl}_3$ webs hinged on the fact that there are exactly three ways to strand an $\mathfrak{sl}_3$ web edge --- one that increases strand depth, one that decreases it, and one that preserves it. In this case, the choice of stranding compatible with integral depth was then forced as illustrated in Figure~\ref{fig:sl3-edge-labels}. For $\mathfrak{sl}_n$ webs with $n\geq 4$, the situation is more complicated. For instance, consider a weight-$1$ edge $e$ whose binary label has its single $1$ in position $j$, for any $2\leq j\leq n-1$. In this case, $e$ is stranded with a simple $(j-1)$-strand directed against $e$ and a simple $j$-strand directed with $e$. Therefore, strand depth on both sides of $e$ is the same. There are $n-2$ such strandings of this form, and so there are multiple choices whenever $n\geq 4$. Thus integral depth alone does not determine a unique stranding when $n\geq 4$, and a finer statistic on webs would be needed to recover an algorithm of the same flavor.
\end{remark}

\subsection{Strand reversal and edge-Kempe equivalence}\label{subsection: strand reversal}

The stranding algorithm of Theorem~\ref{thm:sl3 stranding} fails for $\mathfrak{sl}_3$
webs with flat vertices. This motivates the question raised in the introduction: is
there a sequence of strand reversals (see Figure~\ref{fig: strand reversal example})
connecting any two strandings of a web? We answer this for $\mathfrak{sl}_3$, showing
such a sequence always exists, though we do not give an algorithm for finding one. In
particular, a leading term stranding is reachable from any starting stranding.

\begin{definition}[Edge-Kempe equivalence]\label{def: edge-Kempe equivalence}
Recall that a proper edge-$n$-coloring of a graph is an assignment of one of $n$ colors to each edge so that no two edges sharing a vertex receive the same color.
\begin{itemize}
    \item Given a proper edge-$n$-coloring and a choice of two of the $n$ colors, an {\bf edge-Kempe chain} is a maximal connected subgraph using only edges of those two colors; since the coloring is proper, every vertex has degree at most two in this subgraph, so an edge-Kempe chain is either a path or a cycle.
    \item An {\bf edge-Kempe switch} swaps the two colors along an edge-Kempe chain, producing a new proper edge-$n$-coloring.
    \item Two colorings are {\bf edge-Kempe equivalent} if one can be obtained from the other by a sequence of edge-Kempe switches.
\end{itemize}
\end{definition}

\begin{theorem}[belcastro--Haas, {\cite[Corollary 4.3]{belcastroHaas}}]\label{thm: belcastro haas}
Any two proper edge-$3$-colorings of a planar, bipartite, cubic graph are edge-Kempe equivalent.
\end{theorem}

\begin{remark}
The $3$-connected case of Theorem~\ref{thm: belcastro haas} goes back to Fisk \cite{Fisk}. Robert \cite[Theorem 2.14]{Robert3col} gives a proof of Theorem~\ref{thm: belcastro haas} in the language of closed $\mathfrak{sl}_3$ webs, appealing to the fact that every such web contains a vertexless loop or a face with two or four sides.
\end{remark}

The graphs underlying $\mathfrak{sl}_3$ webs are planar and bipartite but not necessarily cubic, since boundary vertices are univalent, so Theorem~\ref{thm: belcastro haas} does not apply to them directly. We first generalize Theorem~\ref{thm: belcastro haas} to graphs with univalent boundary vertices and then apply this generalization to $\mathfrak{sl}_3$ webs.

\begin{proposition}\label{prop: kempe general}
Let $H$ be a bipartite half-plane graph with trivalent interior vertices and univalent boundary vertices. Then any two proper edge-$3$-colorings of $H$ are edge-Kempe equivalent.
\end{proposition}

\begin{proof}
Let $\overline{H}$ denote $H$ with the two classes of its bipartition swapped. Reflect $\overline H$ across the boundary axis and glue each boundary vertex of $H$ to the corresponding boundary vertex of $\overline H$, matched via the left-to-right order along the boundary axis, erasing these vertices so that each matched pair of boundary edges fuses into a single edge. Call the resulting graph $H'$. Since the gluing always pairs a vertex of one bipartition class with a vertex of the other, $H'$ is well-defined and inherits a bipartition from $H$ and $\overline H$. As $H'$ has no boundary vertices, it is cubic, and it inherits planarity from $H$.

Given a proper edge-$3$-coloring $c$ of $H$, let $\overline c$ be the same coloring transported to $\overline H$ under the identification of edges, and let $c'$ be the coloring of $H'$ that restricts to $c$ on $H$ and $\overline c$ on $\overline H$; this is well-defined since the color on each pair of glued edges agrees, so $c'$ is a proper edge-$3$-coloring of $H'$. Given two colorings $c_1, c_2$ of $H$, transport them to colorings $c_1', c_2'$ of $H'$ in this way. By Theorem~\ref{thm: belcastro haas}, $c_1'$ and $c_2'$ are edge-Kempe equivalent.

We produce a corresponding sequence of edge-Kempe switches on $H$ relating $c_1$ and $c_2$. Consider an edge-Kempe switch in a sequence taking $c_1'$ to $c_2'$. If the corresponding edge-Kempe chain lies entirely in $\overline H$, ignore this move. If the chain lies entirely in $H$, perform the corresponding edge-Kempe switch on the coloring of $H$. If the chain crosses the boundary axis, it decomposes into a union of arcs on either side of the axis, each of which is an edge-Kempe chain in $H$ or $\overline H$ for the colorings restricted to those graphs. The edge-Kempe chains lying in $H$ are pairwise disjoint, so we may perform edge-Kempe switches along them in any order; doing so reproduces, on $H$, the effect of the original edge-Kempe switch on $H'$. Applying this to the full sequence carrying $c_1'$ to $c_2'$ produces a sequence of edge-Kempe switches on $H$ carrying $c_1$ to $c_2$.
\end{proof}

\begin{corollary}\label{cor: sl3 kempe}
Let $G$ be an $\mathfrak{sl}_3$ web. Any two valid strandings of $G$ are related by a sequence of strand reversals.
\end{corollary}

\begin{proof}
The web graph $G$ is oriented such that every vertex is a source or sink, so it is bipartite. Hence $G$ satisfies the hypotheses of Proposition~\ref{prop: kempe general}.

Each stranding $S$ of $G$ can be thought of as a proper edge-$3$-coloring of $G$ with colors red, blue, and purple (red and blue together). We claim that $(i,j)$-strands of $S$ are edge-Kempe chains of the corresponding edge coloring, so that a strand reversal exactly coincides with an edge-Kempe switch. Consider, say, a $(1,2)$-strand $\gamma$. Since $\gamma$ is a $(1,2)$-strand, every edge of $\gamma$ carries a blue strand: the edge is colored blue when the strand runs with the edge's orientation, and purple when it runs against the edge's orientation (with a red strand then running with the edge). Since the vertices of $G$ are sources and sinks, the direction of the edges of $G$ alternate along $\gamma$, so these colors alternate as well: $\gamma$ is a blue--purple alternating chain. It remains to check maximality: if $\gamma$ is open, it terminates at boundary vertices, which are univalent, so the chain cannot be extended past either end; if $\gamma$ is closed, it is already a closed chain, so maximality is automatic. The same argument applies to $(1,3)$- and $(2,3)$-strands, with red--blue and red--purple in place of blue--purple. Hence strand reversal along a strand of $S$ is exactly an edge-Kempe switch on the corresponding edge coloring.

Given two valid strandings $S_1, S_2$ of $G$, view them as proper edge-$3$-colorings as above; by Proposition~\ref{prop: kempe general}, they are edge-Kempe equivalent. Since strand reversal along an $(i,j)$-strand is exactly an edge-Kempe switch on the corresponding edge-Kempe chain, this sequence translates directly into a sequence of strand reversals taking $S_1$ to $S_2$.
\end{proof}

This implies that the algorithmically constructible base stranding of $G$ defined in \cite[Section 6.2]{RTStranding} and a leading term stranding are related by a sequence of strand reversals. It remains open to give an algorithm that produces a sequence of strand reversals connecting the two. Robert shows that certain strand reversals in closed $\mathfrak{sl}_3$ webs alter his degree in a predictable way \cite[Proposition 2.3]{Robert3col}, which suggests one route to such an algorithm.  For $n \geq 4$, it remains unknown whether such a sequence of strand reversals exists. Indeed for general $n$, Robert asks whether all colorings of a closed web are Kempe equivalent \cite[Remark 2.7(3)]{Robert}.

\section*{Acknowledgements}
We are grateful to Julianna Tymoczko for
her ongoing collaboration on the development of stranding and for sharing the idea of strand reversal with this research group. We thank Emily Hafken for proposing the
strand reversal operation. Heather M. Russell was supported by an AMS--Simons Research
Enhancement Grant for PUI Faculty, a University of Richmond sabbatical
fellowship, and the Budapest Semesters in Mathematics Director's Mathematician
in Residence program. The authors also gratefully acknowledge the support of
summer research funding from the University of Richmond. The authors used Anthropic's Claude Opus and Claude Sonnet models to assist with the organization and editing of this article, to help locate and check bibliographic references, and to convert hand-drawn figures into TikZ code. All mathematical content, proofs, and attributions to prior work are the authors' own.


\def\cprime{$'$}

\end{document}